\documentclass[a4paper,reqno]{article}
\usepackage[english]{babel}
\usepackage{a4wide}
\usepackage{graphicx}
\usepackage{amsmath,amssymb,amsthm, amsgen,mathrsfs}
\usepackage[latin1]{inputenc}
\usepackage{listings}
\usepackage{color}
\usepackage[usenames,dvipsnames]{xcolor}
\usepackage{rotating}
\usepackage{hhline}
\usepackage{subfigure}
\usepackage{multirow}
\usepackage{enumerate}
\usepackage{enumitem}
\usepackage{dsfont}
\usepackage{tikz}
\usepackage{url}
\usepackage{mathtools}
\usepackage[normalem]{ulem}
\usepackage{dsfont}

\newcommand{\bOne}{\mathds 1}

\numberwithin{equation}{section}
\newtheorem{thm}{Theorem}[section]
\newtheorem{prop}[thm]{Proposition}
\newtheorem{lem}[thm]{Lemma}

\newtheorem{cor}[thm]{Corollary}
\newtheorem{defn}[thm]{Definition}

\theoremstyle{definition}
\newtheorem{rem}[thm]{Remark}

\def\weakto{\rightharpoonup}

\newcommand{\R}{\mathbb R}
 
\newcommand{\N}{\mathbb N}
\renewcommand{\P}{\mathbb P}

\newcommand{\E}{\mathbb E}
\newcommand{\e}{\varepsilon}

\newcommand{\Dopp}{D^{\operatorname{opp}}}
\newcommand{\Dsame}{D^{\operatorname{same}}}
\newcommand{\Dsamel}{D^{\operatorname{same}, \ell}}
\newcommand{\dist}{\operatorname{dist}}

\newcommand{\bb}{\mathbf b}
\newcommand{\bB}{\mathbf B}

\newcommand{\bW}{\mathbf W}

\newcommand{\bx}{\mathbf x}
\newcommand{\bX}{\mathbf X}

\newcommand{\bY}{\mathbf Y}

\newcommand{\cF}{\mathcal F}

\newcommand{\cK}{\mathcal K}

\newcommand{\cP}{\mathcal P}

\newcommand{\cU}{\mathcal U}

\newcommand{\cW}{\mathcal W}

\newcommand{\cZ}{\mathcal Z}

\newcommand{\sK}{\mathsf K}

\newcommand{\xto}[1]{\xrightarrow{#1}}
\renewcommand{\div}{\mathop{\mathrm{div}}}
\newcommand{\ol}{\overline}

\newcommand{\tstar}[5]{
\pgfmathsetmacro{\starangle}{360/#3}
\draw[#5] (#4:#1)
\foreach \x in {1,...,#3}
{ -- (#4+\x*\starangle-\starangle/2:#2) -- (#4+\x*\starangle:#1)
}
-- cycle;
}

\long\def\pvm#1{{\color{blue}#1}}
\long\def\map#1{{\color{red}#1}}
\long\def\thomas#1{{\color{purple}#1}}
\long\def\comm#1{{\color{orange}#1}} 
\long\def\comm#1{}
\long\def\pvm#1{}
\long\def\map#1{}
\long\def\thomas#1{}

\usepackage[colorlinks=true, linkcolor=black, citecolor=black, urlcolor=black]{hyperref}

\title{Global existence and mean-field limit for a stochastic interacting particle system of signed Coulomb charges}
\author{Patrick van Meurs%
  \thanks{Faculty of Mathematics and Physics,
  Institute of Science and Engineering,
  Kanazawa University
  Kakuma, Kanazawa, 920-1192,
  Japan; ORCID 0000-0001-8184-5061; \url{pjpvmeurs@staff.kanazawa-u.ac.jp}} 
\and 
Mark A.\ Peletier%
  \thanks{Department of Mathematics and Computer Science, Eindhoven University of Technology, The Netherlands; ORCID 0000-0001-9663-3694; \url{m.a.peletier@tue.nl}}
\and 
Thomas Slangen%
  \thanks{Department of Mathematics and Computer Science, Eindhoven University of Technology, The Netherlands; \url{slangenthomas@gmail.com}}
}
\date{}

\begin{document} 

\maketitle


\begin{abstract}
We study a system of stochastic differential equations with singular drift which describes the dynamics of signed particles in two dimensions interacting by the Coulomb potential. In contrast to the well-studied cases of identical particles that either all repel each other or all attract each other, this system contains both `positive' and `negative' particles. Equal signs repel and opposite signs attract each other; apart from the sign, the potential is the same. 

We derive results on well-posedness of the system, on the type of collisions that can occur, and on the mean-field limit as the number of particles tends to infinity. Our results demonstrate that the signed system shares features of both the fully repulsive and the fully attractive cases. Our proof method is inspired by the work of Fournier and Jourdain (The Annals of Applied Probability, 27, pp.~2807-2861, 2017) on the fully attractive case; we construct an approximate system of equations, establish uniform estimates, and use tightness to pass to limits.
\end{abstract}



\section{Introduction}
\label{s:intro}

We study the interacting particle system given by the set of SDEs
\begin{equation} \label{SDE:full:param}
  d X_t^i = \sigma dB_t^i + \gamma \sum_{j =1 }^N b^i b^j K(X_t^i - X_t^j) dt, \qquad i = 1,\ldots,N,
\end{equation}
where $\bX_t = (X_t^1, \ldots, X_t^N) \in (\R^2)^N$ is the list of the positions in $\R^2$ of each of the $N$ particles, $B_i$ are independent $2$-dimensional Brownian motions, and 
\[
  K : \R^2 \setminus \{0\} \to \R^2, \qquad K(x) = \frac x{|x|^2} = - \nabla \log \frac1{|x|}
\]
is the repulsive Coulomb force. We set $K(0) := 0$ for technical reasons, such as not having to exclude particle self-interactions in the summation in \eqref{SDE:full:param}. $\sigma, \gamma, N > 0$ are parameters; $\sigma$ is the temperature and $\gamma$ is the strength of the interactions. Finally, $b^i \in \{-1,+1\}$ is the fixed sign of particle $i$. Depending on which model \eqref{SDE:full:param} is used for, $b^i$ could be the electric charge of a particle, the rotation direction of a vortex, the orientation of a screw dislocation etc.

Our motivation for studying \eqref{SDE:full:param} is that it is at the frontier of what is known in the literature on similar particle systems. A detailed description follows; here we briefly summarize the state of the art:
\begin{itemize}
  \item for interaction forces which are less singular than $K$ (commonly called ``sub-Coulomb"), a lot is known about well-posedness of solutions, properties thereof and the mean-field limit as $N \to \infty$.
  \item for the `fully repulsive' case (i.e.\ \eqref{SDE:full:param} with $b^i = 1$ for all $i$), it is known that particles do not collide, and that similar properties hold as in the previous bullet.
  \item for the `fully attractive' case (not covered by \eqref{SDE:full:param}, but obtained from it by replacing $b^i b^j$ by $-1$), it is known that particle collisions occur with positive probability. In fact, there is a delicate dependence on the parameters $\sigma, \gamma, N$ and the properties of the solution. It depends on these parameters whether weak solutions exist and how many particles can collide.
\end{itemize}
Our `signed' particle system \eqref{SDE:full:param} sits in the middle between the fully repulsive and fully attractive systems. Yet, we are not aware of any rigorous study on it. Our aim is therefore to start such a study. In particular, we are interested in the following questions: 
\comm{["we are not aware" is backed by the following: [CourcelRosenzweigSerfaty23ArXiv] do fully attractive, and lists the intros of [Ser20, NRS22, dCRS23] for overview on related models, and [CD21, Gol22] for surveys. But, [Ser20, NRS22, dCRS23] don't cite FJ17 or CP16, which makes it reasonable to assume that our signed case is not covered yet by this community.]}
\begin{enumerate}
  \item Do (weak) solutions exist?
  \item Is there a sense of uniqueness for solutions?
  \item (a) Do collisions occur? (b) If so, how many particles can collide with each other at the same time? 
  \item Can we establish a mean-field limit as $N \to \infty$?
\end{enumerate} 
The question of uniqueness is very subtle; Fournier and Jourdain~\cite{FournierJourdain17} discuss this in some detail for the fully attractive case, and we have nothing new to add about this. For the other three questions we provide partial answers in the rest of this paper. 

\subsection{Applications} 

In general, particle systems with \textit{signed} particles are ubiquitous in physics: they may represent, for instance, electrons, protons, electrically charged particles, vortices and dislocations. The particle system \eqref{SDE:full:param} is closely related to models for vortices and dislocations. In the case of vortices, the only difference is that $K$ is commonly rotated by $90^\circ$ (Biot-Savart law). In the case of dislocations, \eqref{SDE:full:param} can be of direct use; we motivate this in more detail next. 

Dislocations are line defects in the ($3$-dimensional) crystallographic lattice of a metal; see the textbooks \cite{HirthLothe82,HullBacon11}. Their collective dynamics is the main driving mechanism of plasticity. However, deriving plasticity models from dislocation dynamics is a longstanding open problem. The main reason for this is that the interaction between dislocation lines is nonlocal and singular. This has led to the commonly used modelling assumption that all dislocations are straight and parallel, so that the dislocations can be characterized as points in a $2$-dimensional cross section; see e.g.\ \cite{HudsonVanMeursPeletier20} and the literature overview therein. \eqref{SDE:full:param} is one of the possible descriptions of such a system. The noise in \eqref{SDE:full:param} is an idealized model for the dynamics generated by the thermal vibrations of the atoms around the dislocations. The main goal for the application to dislocations is to understand the properties of \eqref{SDE:full:param} and to pass to the limit $N \to \infty$. These goals are aligned with the aim of this paper.

\subsection{Related literature} 

The system \eqref{SDE:full:param} is interesting because it addresses precisely the critical Coulomb case. Indeed, many results (including much stronger ones than what we hope for in the present paper) hold for interaction forces with singularities strictly weaker than the Coulomb one; see \cite{JabinWang17} for a review and the recent works \cite{JabinWang18} on vortices, \cite{RosenzweigSerfaty23} for convergence rates as $N \to \infty$, \cite{BermanOnnheim19} on propagation of chaos, and~\cite{HoeksemaHoldingMaurelliTse24} on large deviations. In the present paper we will see that the Coulomb singularity is indeed critical in the sense that the system has qualitatively different behaviour depending on the parameters  $\sigma, \gamma$, and $N$.

Within the literature on Coulomb interactions, the best understood scenario is the fully repulsive case. For this case, in the deterministic setting (i.e.\ $\sigma = 0$) Serfaty and Duerinckx manage in \cite{Duerinckx16,DuerinckxSerfaty20} to establish the mean-field limit $N \to \infty$. For the stochastic setting $\sigma > 0$, \cite{LiuYang16} establishes existence and uniqueness of strong solutions and proves the convergence of it to the mean-field limit. These results were extended to three dimensions in \cite{LiLiuYu19}. When a quadratic confinement potential is added, the system is studied in random matrix theory; see e.g.\ \cite{BolleyChafaiFontbona18}.  
\comm{[ \\ 
LY16: see Thm 2.1 on p.4 and Thm 5.2 on p.31 \\
LLY19: See Thm 2.1 on p.6 (although they say it follows from LY16) and Thm 4.1 on p.23 \\
BCF18: They show that no collisions occur and that strong solutions exist with pathwise uniqueness; all thm 1.1. Section 1.4.6 says something about $N \to \infty$, but a clear MF limit seems out of reach for them.
]}

In the fully attractive case, the particle system is given by
\begin{equation} \label{SDE:FJ:sigma}
  d X_t^i = \sigma dB_t^i - \gamma \sum_{j =1 }^N K(X_t^i - X_t^j) dt, \qquad i = 1,\ldots,N.
\end{equation}
This system is delicate in the sense that the Coulomb attraction is critical with respect to the Brownian noise. Indeed, the Brownian motion causes $X^i$ and $X^j$ to scatter at a rate of order $O(|X^i - X^j|^{-1/2})$, whereas the interaction force attracts them towards each other at a rate of the same order. Thus, one may expect that the behavior of \eqref{SDE:FJ:sigma} depends critically on the values of the parameters $\sigma, \gamma , N$, and this is indeed what has been discovered.

Regardless of the values of the parameters, collisions occur with positive probabilty. Hence, the singularity of $K$ at $0$ is visited, which complicates existence and uniqueness of solutions to \eqref{SDE:FJ:sigma}. Despite these challenges, \cite{FournierJourdain17} and \cite{CattiauxPedeches16} have established existence of weak solutions for which the drift is a.s.\ integrable in time. They also prove that particle collisions happen with positive probability, and establish a mean-field limit; each of these results is limited to a region of the parameter space $(\sigma, \gamma , N)$ in which $\gamma$ is sufficiently small with respect to $\sigma$ and $N$. \cite{BreschJabinWang23}  extend parts of these results to more general potentials $K$,
\cite{YangMin21}  to higher spatial dimensions for the particle positions, while
\cite{CourcelRosenzweigSerfaty23ArXiv} extends the mean-field limit to a quantitative propagation of chaos result, and 
\cite{FournierTardy23ArXiv} extends the solution concept to a weaker one based on Dirichlet forms in which collisions between 3 or more particles can be captured. 

\subsection{Main results} 

Our main results give partial answers to the four questions posed above. Before giving these answers, we note that by rescaling time in \eqref{SDE:full:param} we can reduce the parameters $(\sigma, \gamma)$ to the single parameter $\gamma \sigma^{-2}$ without changing the qualitative behaviour of the system (see Section~\ref{s:intro:scal}). Consequently, our main results depend only on $(\sigma, \gamma)$ through the value of $\gamma \sigma^{-2}$. Consequently, our main results depend only on $(\sigma, \gamma)$ through the value of $\gamma \sigma^{-2}$, or put differently, it is only the relative size of $\gamma$ and $\sigma^2$ that matters.

Formally, the answers to our four questions are as follows (see the cited theorems for rigorous statements):
\begin{enumerate}[label=(\roman*)] 
  \item \label{res:for:exist} (Theorem \ref{t:ex}) If $\gamma < \frac12 \sigma^2$, then for any $N \geq 1$ a global weak solution $\bX_t$ exists with the following properties:
  \begin{itemize}
      \item along $\bX_t$ the drift is integrable in time, and
      \item all collisions that happen along $\bX_t$ are simple, i.e.\ collisions of two particles of opposite sign.
    \end{itemize}  
  \item \label{res:for:pos:prob} (Theorem \ref{t:coll}) For any choice of $\sigma, \gamma , N$, if a weak solution contains at least two particles of opposite sign, then on any open time interval there is a nonzero chance that a collision happens.
  \item \label{res:for:mf} (Theorem \ref{t:mf-limit}) If $\gamma = \gamma_N$ with $N \gamma_N \to \overline \gamma  < \frac12 \sigma^2$ as $N \to \infty$, then the following mean-field limit holds: for a sequence $(\bX^N)_N$ of weak solutions with corresponding signs $(\bb^N)_N$, the empirical measures
  \[
    L^{N,+} := \frac1N \sum_{i :\, b^{N,i} = +1} \!\!\!\!\!\delta_{X^{N,i}},
    \qquad   
    L^{N,-} := \frac1N \sum_{i :\, b^{N,i} = -1} \!\!\!\!\!\delta_{X^{N,i}}
  \]
  converge along a subsequence to densities $\rho^+$ and $\rho^-$ which satisfy 
\begin{equation} \label{GB:diffu}
   \left\{ \begin{aligned}
      &\partial_t \rho^+_t + \ol\gamma \div \bigl[\rho^+_t K{*}(\rho^+_t-\rho^-_t)\bigr] = \sigma \Delta \rho^+_t,\\
  &\partial_t \rho^-_t - \ol\gamma \div \bigl[\rho^-_t K{*}(\rho^+_t-\rho^-_t)\bigr] = \sigma \Delta \rho^-_t,
    \end{aligned} \right. 
  \end{equation}   
where ``$*$" denotes the convolution in $\R^2$ over the spatial variable. 
\end{enumerate}   
These results reveal the following properties of \eqref{SDE:full:param}:
\begin{itemize}
  \item The existence result and the mean-field limit require the interaction strength $\gamma$ to be small enough with respect to the temperature $\sigma$. This leans in the parameter space $(\sigma, \gamma , N)$ towards the sub-Coulomb case, and overlaps with the part of the parameter space in which the results of \cite{FournierJourdain17} for the fully attractive case hold (see below).
  
  \item The required upper bound in \ref{res:for:mf} on $\gamma$ of order $O(N^{-1})$ is the natural scaling for the mean-field limit in \eqref{GB:diffu} in which both the interaction and diffusion appear. Hence, the required upper bound in \ref{res:for:exist} on $\gamma$ is rather weak by comparison. 
  
  \item The result \ref{res:for:pos:prob} complements \ref{res:for:exist} by stating that collisions occur with nonzero probability. Since the drift has to be integrable when $\gamma < \frac12 \sigma^2$, this demonstrates that particles instantaneously separate after collision, and that they scatter sufficiently fast after collision. Moreover, \ref{res:for:pos:prob} holds on the full parameter space and for any possible weak solution.
\end{itemize}

\subsection{Outline of the proof}
We largely follow the proofs in \cite{FournierJourdain17}. 
To prove \ref{res:for:pos:prob}, we consider those two particles $X^i$ and $X^j$ of opposite sign which are initially closest together, use Girsanov's Theorem to construct a probability space in which $X^i$ and $X^j$ evolve independently of the other particles until a certain finite stopping time $\tau$, and then show that $X^i$ and $X^j$ collide with positive probability before time~$\tau$. Yang and Min~\cite{YangMin21} apply an alternative, faster proof method  based on the equation satisfied by $\sum_{i=1}^N |X^i_t|^2$. In our setting, however, this proof method only works in the part of the parameter space where $\gamma>\sigma^2$. 
 
When proving \ref{res:for:exist}, the main concern is the integrability of the drift. We obtain it by proving that the drift is bounded in expectation. The method follows that of~\cite{FournierJourdain17} and is based on the following formal computation. Applying It\^o's lemma without considering the differentiability requirements, we get for the distance between two particles (for simplicity, we take particle labels $1$ and $2$) 
\begin{align*}
  d |X_t^1 - X_t^2|
  &= \sigma \frac{X_t^1 - X_t^2}{|X_t^1 - X_t^2|} d(B_t^1 - B_t^2)
    + \frac{\sigma^2 + (b^1 b^2) 2 \gamma}{|X_t^1 - X_t^2|} dt \\
  & \quad + \gamma \frac{X_t^1 - X_t^2}{|X_t^1 - X_t^2|} \sum_{k=3}^N \big( b^1 b^k K(X_t^1 - X_t^k) - b^2 b^k K(X_t^2 - X_t^k) \big) dt.
\end{align*}
The second term on the right-hand side describes the interaction between $X^1$ and $X^2$, including the It\^o correction. By taking the expectation and integrating in time, this term is part of the expectation of the drift which we need to bound. In addition, it has a prefactor of $\sigma^2 + (b^1 b^2) 2 \gamma$. We need this prefactor to be positive irrespective of $b^1 b^2$; this motivates the required bound in \ref{res:for:exist}. In addition, we need to control the term with all the other interactions. With several manipulations this can be done  with $N$-dependent constants. For the mean-field result in \ref{res:for:mf}, however, we need an estimate which is independent of $N$. Then, we need to take away a large bite from the prefactor $\sigma^2 + (b^1 b^2) 2 \gamma$, which leads to the stronger requirement in \ref{res:for:mf}.

To make this procedure rigorous, we introduce a regularized version of \eqref{SDE:full:param} for which  the above computations are allowed and provide an a priori estimate on the drift. Then, we can pass to the limit as the regularization parameter vanishes.

\subsection{Discussion}

While we largely follow the proofs in \cite{FournierJourdain17}, we stress that the introduction of signs in this paper creates a fundamental difference in behaviour. Figure~\ref{fig:difference-grav-signed} ill`ustrates the essence of this difference. In the left part the close pair of two particles exerts twice the attractive force on a third particle that a single particle does; this causes particle clustering (see \cite{FournierJourdain17,FournierTardy23ArXiv} for a detailed discussion). In the right part, the close pair of two particles hardly exerts any force on a third particle. Indeed, to leading order the pairwise interaction forces cancel, causing a net zero attraction. Then, the scattering of the Brownian motion dominates, and no clustering occurs. This fundamental difference results in a much larger parameter range (namely $\gamma < \frac12 \sigma^2$) for the global well-posedness of \eqref{SDE:full:param} than the parameter range 
\begin{equation} \label{gam:range:FJ}
  \gamma 
  \leq 2\sigma^2 \frac{N-2}{N(N-1)}
  = \frac2N \sigma^2 (1 + O(\frac1N))
\end{equation}
for global well-posedness in the fully attractive case. To verify \eqref{gam:range:FJ}, we note that \cite{FournierJourdain17} takes
\[
   \sigma = \sqrt 2, \qquad \gamma = \frac\chi{2\pi N}
\]
and requires for global existence in \cite[Theorem 7]{FournierJourdain17} that 
$\chi \leq 8\pi \frac{N-2}{N-1}$. Recalling that the effective parameter is $\gamma / \sigma^2$, we obtain \eqref{gam:range:FJ}.

\begin{figure}[ht]
  \newcommand{\subfigwidth}{5cm}
  \newcommand{\figwidth}{4cm}
\centering
\begin{tikzpicture}[scale = 1]
    \def \r {.25}
    \def \lgray {black!20!white}
    
    
    
    \begin{scope}[shift={(0,0)}] 
        \draw (0, 0) circle (\r); 
        \fill (0, 0) circle (.4*\r); 
        \draw (0,\r) node[above]{$X^3$};
	\end{scope}
	
	\begin{scope}[rotate = 7] 
        \draw[dotted] (\r,0) -- (2.5 - \r,0); 
        \draw[->, very thick] (\r,0) -- (2,0);
	    \begin{scope}[shift={(2.5,0)}] 
	        \draw (0, 0) circle (\r); 
	        \fill (0, 0) circle (.4*\r); 
	        \draw (\r,0) node[right]{$X^1$};
		\end{scope}        
	\end{scope}
	
	\begin{scope}[rotate = -7] 
	    \draw[dotted] (\r,0) -- (2.8 - \r,0); 
        \draw[->, very thick] (\r,0) -- (1.8,0); 
	    \begin{scope}[shift={(2.8,0)}] 
	        \draw (0, 0) circle (\r); 
	        \fill (0, 0) circle (.4*\r); 
	        \draw (\r,0) node[right]{$X^2$};
		\end{scope}        
	\end{scope}
	
	\draw (2.65/2, -.7) node[below] {\begin{tabular}{c} Fully attractive case: \\ $X^3$ is doubly attracted \end{tabular}};
    
    
    \begin{scope}[shift={(7,0)}]
	    \begin{scope}[shift={(0,0)}] 
	        \draw[red] (0, 0) circle (\r); 
	        \draw[red] (-.6*\r,0) -- (.6*\r,0);
	        \draw[red] (0,-.6*\r) -- (0,.6*\r);
	        \draw (0,\r) node[above]{$X^3$};
		\end{scope}
		
		\begin{scope}[rotate = 7] 
	        \draw[dotted] (\r,0) -- (2.5 - \r,0); 
	        \draw[->, very thick] (\r,0) -- (2,0);
		    \begin{scope}[shift={(2.5,0)}, rotate = -7] 
		        \draw[blue] (0, 0) circle (\r);
		        \draw[blue] (-.6*\r,0) -- (.6*\r,0);
		        \draw (\r,0) node[right]{$X^1$};
			\end{scope}        
		\end{scope}
		
		\begin{scope}[rotate = -7] 
		    \draw[dotted] (\r,0) -- (2.8 - \r,0); 
	        \draw[->, very thick] (-\r,0) -- (-1.8,0); 
		    \begin{scope}[shift={(2.8,0)}, rotate = 7] 
		        \draw[red] (0, 0) circle (\r); 
		        \draw[red] (-.6*\r,0) -- (.6*\r,0);
		        \draw[red] (0,-.6*\r) -- (0,.6*\r);
		        \draw (\r,0) node[right]{$X^2$};
			\end{scope}        
		\end{scope}
		
	    \draw (2.65/2, -.7) node[below] {\begin{tabular}{c} Signed case: \\ $X^3$ is hardly attracted \end{tabular}};
    \end{scope}
\end{tikzpicture}
  \caption{Difference between the fully attractive case and the signed case. Close pairs are stronger attractors than single particles in the fully attractive case, but have hardly any effect in the signed case. }
  \label{fig:difference-grav-signed}
\end{figure}
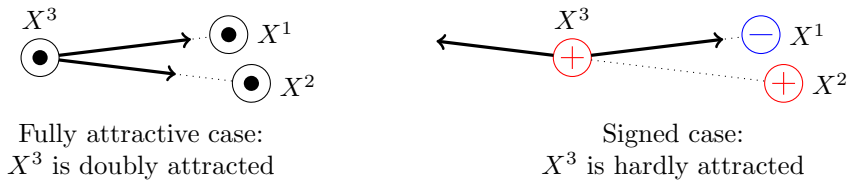

Next, we use our results to answer parts of the four questions raised above:
\begin{itemize} 
  \item Answer to question 3(a): if a weak solution exists and if there are at least two particles with opposite sign, then collisions occur with positive probability anywhere in the parameter space $(\sigma, \gamma , N)$.
  
  \item Answers to questions 1 and 3(b): if $\gamma < \frac12 \sigma^2$, then a weak solution exists for all time. Moreover, this weak solution only exhibits collisions between precisely two particles of opposite sign.
  
  \item Answer to question 4: if $\gamma = \gamma_N$ with $\lim_{N \to \infty} N \gamma < \frac12 \sigma^2$, then \eqref{SDE:full:param} possesses a mean-field limit.
\end{itemize} 
With respect to the literature on the fully repulsive and fully attractive cases, these answers reveal that \eqref{SDE:full:param} shares features of both cases. Indeed, if there are at least two particles with opposite sign, then collisions occur with positive probability, which is a distinct feature of the fully attractive case. Yet, no other type of particle collisions occur, which is a distinct feature of the fully repulsive case.

\subsection{Outlook}

We intend our results on \eqref{SDE:full:param} to pave the way to find more complete answers to the four questions. Here we describe three directions which may lead to these answers.

\paragraph{Mean-field limit for $\frac1N \ll \gamma \ll 1$.} In the range $\frac1N \ll \gamma \ll 1$ we have global existence of weak solutions, but no mean-field limit result. Upon rescaling \eqref{SDE:full:param} such that the prefactor of the drift changes from $\gamma$ to $\frac1N$, the prefactor of the noise becomes $\sigma / \sqrt{N \gamma}$ (see Section \ref{s:intro:scal}), and we expect the mean-field limit (if it exists) to be
\begin{equation} \label{GB}
  \left\{ \begin{aligned}
      &\partial_t \rho^+_t + \div \bigl[\rho^+_t K{*}(\rho^+_t-\rho^-_t)\bigr] = 0,\\
  &\partial_t \rho^-_t - \div \bigl[\rho^-_t K{*}(\rho^+_t-\rho^-_t)\bigr] = 0.
    \end{aligned} \right.
\end{equation}
In the literature on dislocations, these equations were first formally derived in \cite{GromaBalogh99} from \eqref{SDE:full:param} with $\sigma=0$. System \eqref{GB} serves as the basis for more advanced models developed since; see \cite{GromaZaiserIspanovity16}  and the references therein. A rigorous solution concept to \eqref{GB} has been developed in \cite{CannoneElHajjMonneauRibaud10}. In \cite{GarroniVanMeursPeletierScardia19} these equations appear as the many-particle limit from a modified version of \eqref{SDE:full:param} which is deterministic (i.e.\ $\sigma = 0$) and in which $K$ is regularized by a parameter which vanishes as $N \to \infty$. Hence, it would be interesting to discover whether \eqref{GB} can appear as a mean-field limit of \eqref{SDE:full:param}.
\comm{[GZI16: They add yet again a new term to \eqref{GB}, and refer to earlier upgrades to \eqref{GB} by themselves and others]}

\paragraph{Beyond $\gamma < \frac12 \sigma^2$.} The main reason that we do not have any results for $\gamma \geq \frac12 \sigma^2$ is the loss of integrability of the drift. The same issue was noted in \cite{FournierJourdain17}. However, for the special case of two particles (i.e.\ $N=2$) this issue can be side-stepped; see \cite[Theorem 19]{FournierJourdain17}. We revisit these results here for three reasons:
\begin{enumerate}
  \item the system \eqref{SDE:full:param} coincides with the fully attractive case \eqref{SDE:FJ:sigma} for $N=2$ and $b^2 = - b^1$,
  
  \item we rely on the two-particle system in our proof of Theorem \ref{t:coll}, and
  
  \item we expect the two-particle system to be a prototype for handling the collisions in \eqref{SDE:full:param}, which appear to be two-particle collisions only.
\end{enumerate}
For the two-particle system $(X^1, X^2)$, the distance $Y := |X^1 - X^2|$ between the particles is a Bessel process. For Bessel processes there is an established sense for global existence and uniqueness of solutions. \cite[Theorem 19]{FournierJourdain17} uses this concept to reconstruct from $Y$ a certain solution $(X^1,X^2)$, say a ``Bessel-solution", of the two-particle system. The Bessel-solutions need not have an integrable drift, and therefore it need not satisfy the SDE system \eqref{SDE:full:param} in the classical sense. We revisit the properties of Bessel-solutions in detail in Sections~\ref{s:intro:Bessel} and~\ref{s:intro:N2}. Here, we briefly note that:
\begin{itemize}
  \item there is a sense of uniqueness for Bessel-solutions, whereas we are not aware of any uniqueness result on weak solutions of \eqref{SDE:full:param},
  \item the value $\gamma = \frac12 \sigma^2$ is \emph{not} critical for the properties of Bessel-solutions. The only feature which changes is that for $\gamma > \frac12 \sigma^2$ the drift is indeed not integrable on $[\tau, \tau + h]$ a.s., where~$\tau$ is a collision time and $h > 0$ is arbitrary and fixed,
  \item the value $\gamma = \sigma^2$ is critical for whether $X^1$ and $X^2$ separate instantaneously after a collision. For $\gamma \geq \sigma^2$, the Bessel-solution does not separate after the first  collision, say at $\tau$, but satisfies $X_t^1 = X_t^2$ for all $t \geq \tau$. 
\end{itemize}

This sparks the question of whether \eqref{SDE:full:param} can be recast into a form for which Bessel processes can be used to construct solutions. We leave this for future research. 

\paragraph{Other collision rules.} At least for  $\gamma < \frac12 \sigma^2$ the formulation of \eqref{SDE:full:param} implicitly encodes that particles immediately separate after collision. It depends on the application whether this is a desired modeling choice. For instance, in the application to dislocations~(see e.g.\ \cite{Groma97,RoyPeerlingsGeersKasyanyuk08,GeersPeerlingsPeletierScardia13,ScardiaPeerlingsPeletierGeers14,ChapmanXiangZhu15}) it is common to represent systems of straight and parallel edge dislocations by points in a cross-sectional plane. From this point of view it can be natural to remove particles of opposite sign upon collision (this is the basis for the analysis in \cite{PatriziSangsawang21,VanMeursMorandotti19,VanMeursPeletierPozar22,PatriziVanMeurs24}). On the other hand, one might also argue that the assumption of straight and parallel dislocations is an idealisation, and in reality thermal fluctuations prevent annihilation from happening.

In this paper we do not consider any such annihilation rule. Nevertheless, equipping \eqref{SDE:full:param} with an annihilation rule makes the question of the existence of a mean-field limit very interesting. Indeed, a possible limiting equation was studied in \cite{AmbrosioMaininiSerfaty11}, but within those results it remains unclear how much of the particle density is actually annihilated. 

Another reason for considering annihilation as the collision rule is that it naturally matches with the speculated Bessel-solutions of \eqref{SDE:full:param} in the case $\gamma \geq \sigma^2$. Indeed, in this case, if two particles collide, say $X^1$ and $X^2$, then they stick together. Since both particles have opposite sign, their contribution to the drift on each of the other particles $X^3,\ldots,X^N$ vanishes. Moreover, the drift on $X^1$ is the opposite of the drift on $X^2$, and thus the pair $X^1 =  X^2$ behaves as an independent Brownian motion. Hence, $X^1$ (and $X^2$) will not collide with any other particle a.s., and thus the system behaves as if $X^1$ and $X^2$ have been removed. 
\smallskip

The  paper is organized as follows. In Section \ref{s:prel} we provide the preliminaries, which most notably recall some properties of Bessel processes and solve \eqref{SDE:full:param} in the special case of two particles with opposite sign. Sections \ref{s:positive-probability-of-collisions}, \ref{s:global-existence} and \ref{s:MFlimit} concern the statements and proofs of our main three results, Theorems \ref{t:coll}, \ref{t:ex} and \ref{t:mf-limit}, respectively.

\section{Preliminaries}
\label{s:prel}  

We start by making precise several formal descriptions from the introduction. By convention we take $(\Omega, \cF, \P)$ to be a probability space which is large enough to accommodate countably many independent processes of the form \eqref{SDE:full:param}. Whenever no further properties of this space are required, we do not mention it in our analysis.

Next, the precise description of \eqref{SDE:full:param} is the `integrated version' given by
\begin{equation} \label{SDE:integrated}
  X_t^i = X_0^i + \sigma \int_0^t dB_s^i + \gamma \sum_{j =1 }^N b^i b^j \int_0^t K(X_s^i - X_s^j) ds, \qquad i = 1,\ldots,N,
\end{equation}
where $\bX_0$ is a given (random) initial condition.

We use two descriptions for the signed particles. The first is as in the introduction, where we describe a particle in terms of its index $i$ and its position $X^i \in \R^2$. The sign is retrieved from the fixed list $\bb$ by $b^i \in \{\pm 1\}$. We use this description to prove our main Theorems \ref{t:coll} and \ref{t:ex} in which $N$ is fixed. In the second description a signed particle is the position-sign couple
\begin{equation*}
  Y^i := (X^i, b^i) \in \R^2 \times \{\pm 1\} =: \R_\pm^2.
\end{equation*}
In this description we interpret $b^i \in \{\pm 1\}$ as a random variable which is fixed in time. We use this description to prove the mean-field limit in Theorem \ref{t:mf-limit}.

The concept of exchangeability mentioned in the introduction fits naturally to the random-$\bb$ description; a list of signed particles $\bY \in (\R_\pm^2)^N$ is exchangeable if $P_{\bY} = P_{\bY^s}$ for any permutation~$s$ ($(\bY^s)^i := Y^{s(i)}$). In the deterministic-$\bb$ description, we use a modified notion of exchangeability: given $\bb \in \{\pm 1\}^N$, we say that $\bX$ is exchangeable with respect to $\bb$ if $P_{\bX} = P_{\bX^s}$ for any permutation~$s$ which satisfies $\bb = \bb^s$.
 
Next we introduce the weak solution concept to solutions of \eqref{SDE:full:param} that we use in the remainder of the paper. We state it in the deterministic-$\bb$ description; the version for the random-$\bb$ description follows from it with obvious modifications.
 
\begin{defn}[Weak solution] \label{d:wSol}
Let $T > 0$, $\bX_0 \in \R^{2n}$ and $\bb \in \{\pm 1\}^N$ be given. An $\R^{2N}$-valued process $\bX$ is a weak solution to the SDE in \eqref{SDE:full:param} on $[0,T)$ with initial condition $\bX_0$ if there exists a $2N$-dimensional Brownian motion $\bB$ such that
\begin{itemize}
  \item $\bB$ and $\bX$ are adapted to the same filtration $\cF_t$,
  \item $\bX$ is path-continuous a.s., 
  \item $\displaystyle \int_0^t \sum_{i=1}^N \bigg| \sum_{j=1}^N K(X_s^i - X_s^j) \bigg| \, ds < \infty$ a.s.\ for all $t \in(0, T)$, and
  \item the integrated SDE in \eqref{SDE:integrated} holds a.s. 
\end{itemize}
\end{defn} 
 
\begin{prop}[Quick observations] \label{p:obv}
If $\bX$ is a weak solution to \eqref{SDE:full:param}, then:
\begin{enumerate}[label=(\roman*)]
   \item \label{p:obv:exch} exchangeability of $\bX$ with respect to $\bb$ is conserved in time,
   \item \label{p:obv:bflip} $\bX$ remains a weak solution to \eqref{SDE:full:param} if $\bb$ is replaced by $-\bb$,
   \item \label{p:obv:relab} for any permutation $s$, $\bX^s$ is a weak solution to \eqref{SDE:full:param} if $\bX_0$ and $\bb$ are replaced by $\bX_0^s$ and~$\bb^s$.
\end{enumerate}
\end{prop} 

\subsection{Scaling invariance of the particle system}
\label{s:intro:scal}

Consider \eqref{SDE:full:param}. Scaling space by a length scale $L > 0$ and time by a time scale $\alpha$, we set
\[
  \bY_t := \frac1L \bX_{t/\alpha}.
\]  
Then, since $\tilde \bB_t := \sqrt \alpha \bB_{t/\alpha}$ is a $2N$-dimensional Brownian motion and $K(x/L) = L K(x)$, we get
\begin{equation} \label{scaled:SDE} 
  dY_t^i = \frac \sigma{ \sqrt{ \alpha L^2 } } d \tilde B_t^i + \frac{\gamma}{\alpha L^2} \sum_{j =1 }^N b^i b^j K(Y_t^i - Y_t^j) dt.
\end{equation}
Note that the right-hand side only depends on $\alpha$ and $L$ through the factor $\alpha L^2$. This shows that the system \eqref{SDE:full:param} has the following scale invariance: scaling space by a factor $\beta > 0$ leads to the same result as scaling time by a factor $\beta^2$. We often use this to scale space such that in \eqref{SDE:full:param} $\sigma$ and $\gamma$ are respectively replaced by $\sigma / \sqrt \gamma$ and $1$. 
\comm{[In other words, in parameter space $(\sigma, \gamma)$, along the curves defined by $\gamma = C \sigma^2$ for some fixed constant $C > 0$, the processes are equivalent up to a rescaling.]} 

\subsection{Squared Bessel processes}
\label{s:intro:Bessel} 

Bessel processes are well studied; see e.g.\ \cite{Katori16,RevuzYor99}. For the statements in this section we refer to \cite[Chapter XI]{RevuzYor99} unless mentioned otherwise.

Let $\delta \in \R$ be a parameter, $x \geq 0$ a starting position, and $\beta$ a $1$-dimensional Brownian Motion. For $\delta \geq 0$, the SDE
\begin{equation} \label{R:Bessel:SDE}
  dR_t = 2 \sqrt{ R_t } d \beta_t + \delta dt, \qquad R_0 = x
\end{equation}
has a unique strong solution $R$. 
It is called a squared Bessel process of dimension $\delta$ started at $x$. 
The typical example of a squared Bessel process of integer dimension $\delta = n \in \N$ is $| B |^2$, where $B$ is an $n$-dimensional Brownian motion. 

The following lemma describes the well-posedness  for all $\delta$, positive and negative.

\begin{lem}[Squared Bessel processes in positive and negative dimension; {\cite[\S XI.1]{RevuzYor99}}]
  \label{l:sqBessel}
  \noindent
\begin{enumerate}
\item   \label{i:l:sqBessel:pos-dim}
For $\delta\geq0$ strong solutions of \eqref{R:Bessel:SDE}  exist and are unique. 
\item  For $\delta<0$ strong solutions of \eqref{R:Bessel:SDE}  exist and are unique up to the stopping time $\tau = \inf \{ t \geq 0 \mid R_t = 0 \}$. If we understand strong solutions for $\delta<0$ to be frozen when they hit zero, then this type of solutions is unique.
\label{i:l:sqBessel:neg-dim}
\end{enumerate}
\end{lem}
\noindent


As illustrated by the lemma above, squared Bessel processes may hit zero, or not. Loosely speaking,
\begin{enumerate} 
  \item if $\delta \geq 2$ and $x > 0$, then $R$ never hits $0$,
  \item if $\delta < 2$, then with positive probability $R$ hits $0$ on any open time interval. When it hits $0$, say at time $t$, then
  \begin{enumerate}
    \item if $\delta > 0$, then it leaves $0$ instantaneously,
    \item if $\delta \leq 0$, then it is stuck at $0$ on $[t,\infty)$.
  \end{enumerate}
\end{enumerate}
The following theorem makes these statements rigorous.

\begin{thm}[Zero-hitting properties] \label{t:BP}
Let $R$ be a squared Bessel process of dimension~$\delta \in \R$. Then $R$ is a continuous process which satisfies the following properties:
\begin{enumerate}[label=(\roman*)]
  \item \label{t:BP:geq2} If $\delta \geq 2$, then for every $x > 0$ we have\/ $\P (\forall \, t > 0 : R_t > 0 \,|\,R_0 = x) = 1$.
  \item \label{t:BP:less2} Let $a, T > 0$ and assume that $\P(R_0 < a) > 0$. If $\delta < 2$, then 
  \[
    \P \Big( \min_{0 \leq t \leq T} R_t = 0, \ \max_{0 \leq t \leq T} R_t < a \Big) > 0.
  \]
\end{enumerate}
\end{thm}
\noindent
Part~\ref{t:BP:geq2} of Theorem~\ref{t:BP} is \cite[Th.~1.1(i)]{Katori16}, and part \ref{t:BP:less2} is proven in Step 5 of the proof of \cite[Proposition 4]{FournierJourdain17}.

\subsection{The case of two particles}
\label{s:intro:N2}

As described in the Introduction, the special case of two particles has a close connection to Bessel processes. 
We refer to \cite[Section 6]{FournierJourdain17} for proofs of the statements in this section. 

For two particles the SDE system  \eqref{SDE:full:param} becomes
\begin{align} \label{SDE:N2:full:param} 
  \left\{ \begin{aligned}
    d X_t^1 &= \sigma dB_t^1 - \gamma K(X_t^1 - X_t^2) dt, \\
    d X_t^2 &= \sigma dB_t^2 + \gamma K(X_t^1 - X_t^2) dt.  
  \end{aligned} \right.  
\end{align}
Here $\gamma>0$ corresponds to particles of opposite sign and  $\gamma < 0$ of same sign. Also here it is not clear whether the drift is integrable. In fact, for $\gamma > \frac12 \sigma^2$, the drift $\gamma |K(X_t^1 - X_t^2)|$ is a.s.\  not integrable on $(\tau, \tau + h)$, where $h > 0$ and $\tau$ is a collision time.

However, for any choice of $\sigma > 0$ and $\gamma \in \R$ it is possible to construct a meaningful solution concept by changing variables. Let $(X_t^1, X_t^2)$ be the strong solution of \eqref{SDE:N2:full:param} up to the first collision time $\tau$. Let $S := X^1 + X^2$ be the sum  and  $D := X^1 - X^2$ be the difference. Then,
\begin{align} \label{SDE:SD} 
  \left\{ \begin{aligned}
    d S_t &= \sigma d(B_t^1 + B_t^2), \\
    d D_t &= \sigma d(B_t^1 - B_t^2) - 2 \gamma K(D_t) dt.
  \end{aligned} \right.  
\end{align}
Note that the processes $S$ and $D$ are independent and that $(S_t - S_0)/(\sigma \sqrt 2)$ is a Brownian motion. 

The equation for $D$ still has a singular drift, and we therefore apply a further change of variables. The squared distance $R := |D|^2/(2\sigma^2)$ satisfies the SDE
\begin{equation}
  \label{eq:SDE-squared-Bessel-process}
  dR_t = 2 \Big( 1 - \frac\gamma{\sigma^2} \Big) dt + 2 \sqrt{ R_t} d \beta_t,
\end{equation}
where
\[
  \beta_t := \int_0^t \frac{D_s}{|D_s|} \cdot d W_s, 
  \qquad W := \frac{B^1 - B^2}{\sqrt 2} ,
\]
in which $\beta$ is a $1$-dimensional and $W$ a $2$-dimensional Brownian Motion.
Therefore $R$ is a squared Bessel process of dimension $\delta = 2 ( 1 - \frac\gamma{\sigma^2} )<2$. We will use the properties of objects such as $R$ in the proofs of Theorems~\ref{t:coll} and~\ref{t:ex} below.

Fournier and Jourdain~\cite[Theorem 19]{FournierJourdain17}  study yet another change of variables,  $Z := |D|^2 D$ (with inverse $D = |Z|^{-2/3} Z$). They show that $Z$ satisfies an SDE in $\R^2$ with locally bounded drift and degenerate diffusion term. Neither of the two is Lipschitz continuous at $Z = 0$, but \cite[Theorem 19]{FournierJourdain17} guarantees the existence of a global weak solution. For $\gamma > \sigma^2$ the solution exists only up to the first time $\tau$ at which $Z$ hits $0$; after that the process is frozen, which is consistent with the second part of Lemma~\ref{l:sqBessel}.

In conclusion, with processes $S$ and $Z$ obtained above, we can define a notion of solutions $(X_1, X_2)$ to \eqref{SDE:N2:full:param} (referred to as ``Bessel-solutions" in the Introduction) by inverting the changes of variables as follows: take $D := |Z|^{-2/3} Z$, $X_1 := \frac12(S + D)$ and $X_2 := \frac12(S - D)$.

\medskip

The properties of squared Bessel processes mentioned in Section \ref{s:intro:Bessel} give rise to corresponding properties for $(X^1, X^2)$. Figure \ref{fig:BessSol} illustrates these properties. 
\begin{enumerate}
  \item If both particles have the same sign (i.e.\ $\gamma \leq 0$), then $X^1$ and $X^2$ a.s.\ do not collide in finite time.
  \item If $X^1$ and $X^2$ have opposite sign (i.e.\ $\gamma > 0$), then on any open time interval a collision occurs with nonzero probability. Moreover, 
  \begin{enumerate}
    \item if $\gamma < \frac12 \sigma^2 $, then $(X^1, X^2)$ is a weak solution to \eqref{SDE:N2:full:param} in the sense of Definition \ref{d:wSol}, but if $\gamma > \frac12 \sigma^2 $, it is not,
    \item if $\gamma < \sigma^2 $, then $X^1$ and $X^2$ separate instantaneously after a collision occurs, but if  $\gamma \geq  \sigma^2 $, then  $X^1$ and $X^2$ remain stuck together from time $t$ onwards. 
  \end{enumerate}
\end{enumerate}
The regime $\gamma\geq \sigma^2$, in which opposite-sign particles collide and remain stuck together, has a natural interpretation in the context of dislocations: it corresponds to the collision of two defects with opposite signs, which then annihilate each other.  
\comm{Remark: For $D_t$, the radial component $|D_t|$ can be computed indep-y from the angular component, but the angular component depends on $|D_t|$.} 

\begin{figure}[h!]
  \centering
  \begin{tikzpicture}[scale = 1]
    \def \r {.5}
    \def \lgray {black!20!white}
    
    
    \draw[<->] (0,5) node[left]{$\sigma$} -- (0,0) -- (5,0) node[below]{$\gamma$};
    \draw[domain=0:3, smooth, white!20!gray, dashed] plot ({5*\x*\x/9},\x);
    \draw[domain=0:5, smooth, thick] plot ({\x*\x/5},\x);
    \draw (5,5) node[right]{$\gamma = \frac12\sigma^2 $};
    \draw[white!20!gray] (5,3) node[right]{$\gamma = \sigma^2$};
 
    \draw (1.5,4.5) node{\begin{tabular}{c} weak solution \\ $(X^1, X^2)$ \end{tabular} };
    \draw (3,2) node{\begin{tabular}{c} Bessel-solution \\ $(X^1, X^2)$ \end{tabular}};
    
    \begin{scope}[shift={(8,0)}]
	    \draw[<->] (0,5) node[left]{$\sigma$} -- (0,0) -- (5,0) node[below]{$\gamma$};
	    \draw[domain=0:3, smooth, thick] plot ({5*\x*\x/9},\x);
	    \draw[domain=0:5, smooth, white!20!gray, dashed] plot ({\x*\x/5},\x);
	    \draw (5,3) node[right]{$\gamma = \sigma^2$};   
	    \draw[white!20!gray] (5,5) node[right]{$\gamma = \frac12\sigma^2 $};
	    
	    \begin{scope}[shift={(1,3)}, scale = .5] 
		    \begin{scope}[shift={(2,1.5)}]
		      \tstar{.5*\r}{1*\r}{10}{11}{fill=yellow}
		    \end{scope}
		    
		    \foreach \point in {(0,0), (4,3)}{
		    \begin{scope}[shift=(\point),scale=1] 
		        \draw[red] (0, 0) circle (\r); 
		        \draw[red] (-.6*\r,0) -- (.6*\r,0);
		        \draw[red] (0,-.6*\r) -- (0,.6*\r); 
		    \end{scope}
		    }
		    
		    \foreach \point in {(0,3), (4,0)}{
		    \begin{scope}[shift=(\point),scale=1] 
		        \draw[blue] (0, 0) circle (\r);
		        \draw[blue] (-.6*\r,0) -- (.6*\r,0);
		    \end{scope}
		    }  
		    
		    \draw[->,red] (.6,.1) to [out=15,in=240, looseness=1] (1.7, .9);
		    \draw[->,blue] (.6,2.9) to [out=345,in=120, looseness=1] (1.7, 2.1);
		    \draw[<-,blue] (3.4,.1) to [out=165,in=300, looseness=1] (2.3, .9);
		    \draw[<-,red] (3.4,2.9) to [out=195,in=60, looseness=1] (2.3, 2.1);
	    \end{scope}	   
	     
	    \begin{scope}[shift={(3.3,.4)}, scale = .5] 
		    \begin{scope}[shift={(2,1.5)}]
		      \tstar{.5*\r}{1*\r}{10}{11}{fill=yellow}
		    \end{scope}
		    
		    \foreach \point in {(0,0)}{
		    \begin{scope}[shift=(\point),scale=1] 
		        \draw[red] (0, 0) circle (\r); 
		        \draw[red] (-.6*\r,0) -- (.6*\r,0);
		        \draw[red] (0,-.6*\r) -- (0,.6*\r); 
		    \end{scope}
		    }
		    
		    \foreach \point in {(0,3)}{
		    \begin{scope}[shift=(\point),scale=1] 
		        \draw[blue] (0, 0) circle (\r);
		        \draw[blue] (-.6*\r,0) -- (.6*\r,0);
		    \end{scope}
		    } 
		    
		    \foreach \point in {(4,1.5)}{
		    \begin{scope}[shift=(\point),scale=1] 
		        \draw[green!70!black] (0, 0) circle (\r); 
		        \draw[green!70!black] (0, 0) circle (.4*\r);
		    \end{scope}
		    } 
		    
		    \draw[->,red] (.6,.1) to [out=15,in=240, looseness=1] (1.7, .9);
		    \draw[->,blue] (.6,2.9) to [out=345,in=120, looseness=1] (1.7, 2.1);
		    \draw[->,green!70!black] (2.6,1.5) --++ (.8, 0);
	    \end{scope}  
    \end{scope}    
\end{tikzpicture}
  \caption{Areas in the parameter space $(\sigma,\gamma)$ where weak solutions exist (left) and where sticky collisions occur (right).}
  \label{fig:BessSol}
\end{figure}
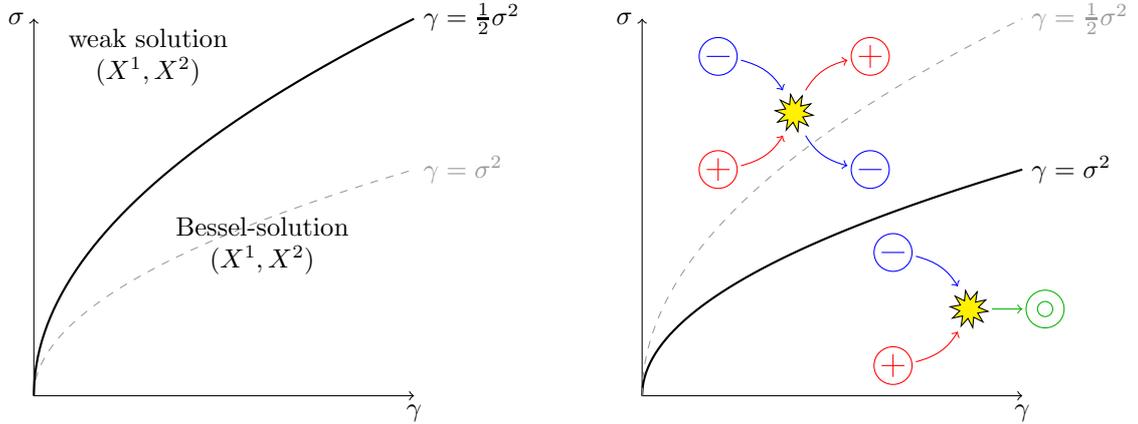

\newpage

\section{Collisions happen with positive probability}
\label{s:positive-probability-of-collisions}

In this section we state and prove Theorem \ref{t:coll} on the possibility of particle collisions.
For a weak solution $\bX$ to \eqref{SDE:full:param}, let 
\begin{equation} \label{Dt}
  D_t := \min_{i \neq j} |X_t^i - X_t^j|  
 \end{equation}
be the shortest distance between any two particles.


\begin{thm}[Collisions happen with positive probability] \label{t:coll} 
Let $N \geq 2$, $T, \sigma, \gamma > 0$, $\bb \in \{-1,1\}^N$ such that $b^i = -b^j$ for some $i,j$, and $\bX_0$ exchangeable with respect to $\bb$ (see Section \ref{s:prel}). Suppose that there exists a weak solution $\bX$ of \eqref{SDE:full:param} on $[0,T]$. Then
\begin{align*}
  \P \Big( \inf_{0 \leq t \leq T} D_t = 0 \Big) > 0.
\end{align*}  
\end{thm}

\comm{Actually, we would like to prove the stronger statement $\P \Big( \inf_{0 \leq t \leq T} \Dopp_t = 0 \Big) > 0$. However, FJ17's proof does not allow for that. Indeed, when reasoning below by contradiction on the $D_t$-statement, we get for free that no $++$-collisions happen, and we need that in the proof. If we were to work with the $\Dopp_t$ statement instead, then we cannot (as of 18 May 2023) exclude such collisions.
}

\begin{proof}[Proof of Theorem \ref{t:coll}]
The proof is a modification of the proof of \cite[Proposition 4]{FournierJourdain17}. The main difference with the setting in \cite{FournierJourdain17} is that in our setting the new object $\Dopp_t$ appears. This requires a more careful setup in Step 1 below. The other steps are similar. 

Let $(\Omega, \cF, \P)$ be a probability space on which $\bX$ is a weak solution. 
Set
\[
  \Dopp_t := \min_{\substack{ i \neq j \\ b^i = -b^j }} |X_t^i - X_t^j|
\]
as the shortest distance between any two particles of \emph{opposite} sign. Clearly, $\Dopp_t \geq D_t$. Without loss of generality we may assume the following:
\begin{itemize}
  \item $\gamma = 1$ (see Section \ref{s:intro:scal}),
  \item $b^1 = 1 = -b^2$ and $|X_0^1 - X_0^2| = \Dopp_0$ (see Proposition \ref{p:obv}),
  \item $N \geq 3$; for $N=2$ the proof below applies with simplifications,
  \item $\bX_0 = \bx_0$ is deterministic, where $\bx_0$ is such that
\begin{equation} \label{pfzr}
  D_0 = \min_{i < j} |x_0^i - x_0^j| > 0.
\end{equation}
Indeed, if $D_0 = 0$, then Theorem \ref{t:coll} is trivial.
\end{itemize}
Because of these assumptions, equation~\eqref{SDE:full:param} reads 
\begin{equation} \label{SDE:pf}
  d X_t^i = \sigma dB_t^i + \sum_{j =1 }^N b^i b^j K(X_t^i - X_t^j) dt
\end{equation}
and $D_0, \Dopp_0 > 0$ are deterministic.

We reason by contradiction. Suppose that   
\begin{equation} \label{pfzq}
  \P \Big( \inf_{0 \leq t \leq T} D_t = 0 \Big) = 0. 
\end{equation}
Then, for almost every $\omega \in \Omega$ the integrated SDE in \eqref{SDE:pf} has a solution $\bX(\omega)$ that exists and is unique up to time $T$. 
\comm{[already for weak sol-ns paths are continuous, and thus they say a positive distance away from collisions. Hence, an $\omega$-ball around the singularity is never visited.]}
\smallskip

\textbf{Step 1:} Preliminary estimates for singling out $X^1$ and $X^2$. The aim in this step is to define a stopping time $\tau \in (0, T]$ before which:
\begin{enumerate}
  \item $X_t^1$ and $X_t^2$ are not much further apart than $|x_0^1 - x_0^2| = \Dopp_0$,
  \item the midpoint $\frac12(X_t^1 + X_t^2)$ is close to the initial midpoint
\[
  \ol x_0 := \frac12 (x_0^1 + x_0^2), \quad \text{and}
\] 
  \item $X_t^1$ and $X_t^2$ are closer to $\ol x_0$
than any other particle.
\end{enumerate}
Below, these conditions are captured by the respective stopping times $\tau_1$, $\tau_2$ and $\tau_3$.

By the definitions of $D_0$ and $\Dopp_0$ and elementary trigonometry (see Figure \ref{fig:x1x2}), 
we have that 
\comm{[p.10 top]} 
\begin{equation} \label{pfzu}
  \min_{3 \leq j \leq N} |x_0^j - \ol x_0| 
  \geq |y - \ol x_0|
  = \frac12 \sqrt{ (\Dopp_0)^2 + 2 D_0^2 },
\end{equation}
where the point $y \in \R^2$ is specified in Figure \ref{fig:x1x2} (it is any of the four intersection points of two circles).
Note that 
\[
  \min_{3 \leq j \leq N} |x_0^j - \ol x_0|
  > \frac12 \Dopp_0 = |x_0^1 - \ol x_0| = |x_0^2 - \ol x_0|.
\]
For technical reasons, we take $b>a > \frac12$ such that
\begin{equation} \label{pfzt}
  \frac12 \Dopp_0 < a \Dopp_0 < b \Dopp_0 < \frac12 \sqrt{ (\Dopp_0)^2 + 2 D_0^2 }.
\end{equation}
Using $a,b$, we define $\tau$ as follows:
\begin{align*}
  \tau_1 &:= T \wedge \inf \Big\{ t \in [0,T] : |X_t^1 - X_t^2| \geq \Big( a + \frac12 \Big) \Dopp_0 \Big\}, \\
  \tau_2 &:= T \wedge \inf \Big\{ t \in [0,T] : |X_t^1 + X_t^2 - 2 \ol x_0| \geq \Big( a - \frac12 \Big) \Dopp_0 \Big\}, \\
  \tau_3 &:= T \wedge \inf \Big\{ t \in [0,T] : \min_{3 \leq j \leq N} |X_t^j - \ol x_0| \leq b \Dopp_0 \Big\}, \\
  \tau &:= \min_{1 \leq k \leq 3} \tau_k.
\end{align*}
By construction and the path-continuity of $\bX$, we have $\tau > 0$ a.s. Moreover, for any $3 \leq j \leq N$ we have on $[0,\tau]$
\begin{align*}
  |X_t^1 - X_t^j|
  &\geq |X_t^j - \ol x_0| - \frac12 |X_t^1 + X_t^2 - 2 \ol x_0| - \frac12 |X_t^1 - X_t^2|  \\
  &\geq \Big( b - \frac12\Big( a - \frac12 \Big) - \frac12\Big( a + \frac12 \Big) \Big) \Dopp_0
  = (b-a) \Dopp_0
  > 0.
\end{align*}
A similar estimate holds for $X_t^2$. Thus, there exists a constant $c > 0$ such that
\begin{equation} \label{pfzw}
  \min_{t \in [0,\tau]}\, \min_{i = 1,2}\, \min_{3 \leq j \leq N} |X_t^i - X_t^j| \geq c.
\end{equation}
\smallskip

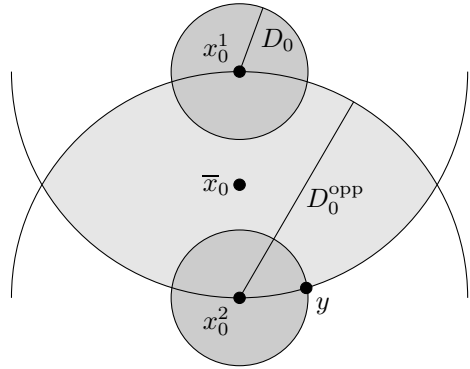
\begin{figure}[h]
\centering
\begin{tikzpicture}[scale=1.5]
  \def \r {.05}   
   
  \begin{scope}[shift={(0,1)},scale=1] 
    \fill[black!10!white] (0,0)-- (210:2) arc (210:330:2) -- cycle; 
  \end{scope}
  
  \begin{scope}[shift={(0,-1)},scale=1] 
    \fill[black!10!white] (0,0)-- (30:2) arc (30:150:2) -- cycle; 
  \end{scope}
  
  \draw[fill = black!20!white] (0,1) circle (.6);
  \draw[fill = black!20!white] (0,-1) circle (.6);  
  
  \begin{scope}[shift={(0,1)},scale=1] 
    \draw (180:2) arc (180:360:2); 
  \end{scope}
  
  \begin{scope}[shift={(0,-1)},scale=1] 
    \draw (0:2) arc (0:180:2); 
  \end{scope}
  
  \draw[fill = black] (0,1) circle (\r) node[above left] {$x_0^1$};
  \draw[fill = black] (0,0) circle (\r) node[left] {$\ol x_0$};
  \draw[fill = black] (0,-1) circle (\r)  node[below left] {$x_0^2$};
  
  \begin{scope}[shift={(0,-1)},rotate=60] 
    \draw[thin] (0,0) -- (2,0) node[midway,right] {$\Dopp_0$}; 
  \end{scope} 
  
  \begin{scope}[shift={(0,1)},rotate=287] 
    \draw [fill = black] (2,0) circle (\r) node[below right] {$y$}; 
  \end{scope}  
  
  \begin{scope}[shift={(0,1)},rotate=70] 
    \draw[thin] (0,0) -- (.6,0) node[midway,right] {$D_0$}; 
  \end{scope} 
  
\end{tikzpicture}
\caption{By the definitions of $D_0$ and $\Dopp_0$, any particle $x_0^j$ other than $x_0^1$ or $x_0^2$ is outside of the gray regions.}
\label{fig:x1x2}
\end{figure} 

\textbf{Step 2:} Decoupling $X^1$ and $X^2$ from the other particles. We use Girsanov's Theorem to remove the dependence between $(X^1, X^2)$ and $(X^3, X^4,\ldots,X^N)$, by changing $\P$ to some equivalent probability measure $\tilde \P$. With this aim, we set 
\comm{[equiv means here  have the same null-sets]}
\[
  Y_t^i := \left\{ \begin{aligned}
    &\frac1{\sigma} \sum_{j = 3}^N b^i b^j K(X_t^i - X_t^j) 
    &&\text{if } i = 1,2,  \\
    &\frac1{\sigma} \sum_{j =1}^2 b^i b^j K(X_t^i - X_t^j) 
    &&\text{if } 3 \leq i \leq N,
  \end{aligned} \right.
\]
and write the SDE system \eqref{SDE:pf} as
\begin{align*}
  d X_t^1 &= \sigma dB_t^1 - K(X_t^1 - X_t^2) dt + \sigma Y_t^1 dt, \\
  d X_t^2 &= \sigma dB_t^2 + K(X_t^1 - X_t^2) dt + \sigma Y_t^2 dt, \\
  d X_t^i &= \sigma dB_t^i + \sum_{ j = 3 }^N b^i b^j K(X_t^i - X_t^j) dt + \sigma Y_t^i dt \qquad \text{for all } 3 \leq i \leq N.  
\end{align*} 
To remove the contribution of $\bY$ up to $\tau$, we consider the stopped process $\bY_t^\tau := \bY_{t \wedge \tau}$. By \eqref{pfzw}, there exists $C>0$ such that $\| \bY^\tau \|_{L^2(0,T)}^2 \leq C T$ a.s. Then, by Novikov's criterion, Girsanov's Theorem applies, and we obtain that
\[
  \bW_t := \bB_t + \int_0^t \bY_s^\tau \, ds \qquad \text{for all } t \in [0,T]
\]
defines a Brownian motion  under some $\P_*$ equivalent to $\P$ and adapted to the same filtration $\cF_t$ to which $\bB_t$ is adapted. 
Then, since the stopped process $\bY_t^\tau$ equals $\bY_t$ on $[0,\tau]$, we can rewrite the SDEs above on $[0,\tau]$ as 
\begin{align} \label{pfzp}
  \left\{ \begin{aligned}
    d X_t^1 &= \sigma dW_t^1 - K(X_t^1 - X_t^2) dt \qquad \text{on } [0, \tau] \\
  d X_t^2 &= \sigma dW_t^2 + K(X_t^1 - X_t^2) dt \qquad \text{on } [0, \tau] 
  \end{aligned} \right.  
\end{align}
and 
\begin{equation} \label{pfzy}
  d X_t^i = \sigma dW_t^i + \sum_{ j = 3 }^N b^i b^j K(X_t^i - X_t^j) dt \qquad \text{on $[0, \tau]$ for all } 3 \leq i \leq N. 
\end{equation}
Since under $\P_*$ the $2$-dimensional Brownian motions $W_t^i$ for $i=1,\ldots,N$ are independent, the two systems above decouple, i.e.\ $(X_t^1, X_t^2)$ is independent of $(X_t^3, \ldots, X_t^N)$ on $[0, \tau]$.

\smallskip

\textbf{Step 3:} Conversion to a Bessel process. 
We recall from Section \ref{s:intro:N2} that \eqref{pfzp} can be rewritten in terms of $S_t := X_t^1 + X_t^2$ and $D_t := X_t^1 - X_t^2$ on $[0,\tau]$. Moreover,  
\begin{equation} \label{pfyx}
  S_t = 2 \ol x_0 + \sigma (W_t^1 + W_t^2) \qquad \text{for all } 0 \leq t \leq \tau
\end{equation}
and $R_t := |D_t|^2 / (2 \sigma^2)$ is on $[0,\tau]$ a squared Bessel process of dimension $2(1 - \sigma^{-2})$ as in~\eqref{eq:SDE-squared-Bessel-process}, i.e.
\begin{equation} \label{pfyy}
  d R_t = 2 \sqrt{R_t} d \beta_t + 2 \Big(1 - \frac1{\sigma^2} \Big) \, dt
  \qquad \text{on } [0, \tau],
\end{equation}
where
\begin{equation*}
  \beta_t := \int_0^t \frac{ X_s^1 - X_s^2 }{|X_s^1 - X_s^2|} \cdot d \Big( \frac{ W_s^1 - W_s^2 }{\sqrt 2} \Big)
  \qquad \text{on } [0, T]
\end{equation*}
is a $1$-dimensional Brownian motion.

We claim that the Brownian motions $\beta_t, \widetilde W_t := (W_t^1 + W_t^2)/\sqrt 2, W_t^3, \ldots, W_t^N$ are independent on $[0,T]$ with respect to $\P_*$. To prove this, we note from Levy's theorem (e.g.~\cite[Theorem 3.16]{KaratzasShreve98}) that it is sufficient to show that all cross variations are $0$. Since $W_t^i$ are independent by construction, it is left to show that the cross variation of $\beta_t$ and any of the other $N-1$ Brownian motions is $0$.  

We start with showing that $[\beta_t, \widetilde W_t]$ is $0$ for all $t \in [0,T]$. Writing
\begin{equation*}
  \beta_t = \int_0^t \Sigma_s^\beta \,d \begin{pmatrix} W_s^1\\W_s^2\end{pmatrix}, 
  \qquad       
  \widetilde W_t = \int_0^t \Sigma^W \,d \begin{pmatrix} W_s^1\\W_s^2\end{pmatrix},
\end{equation*}
where
\begin{equation*}
  \Sigma^\beta_s := \frac1{\sqrt 2}\begin{pmatrix} Z_s \\ -Z_s \end{pmatrix}^\top \in \R^{1\times 4}
\quad\text{{and}}\quad
\Sigma^W := \frac1{\sqrt 2} \begin{pmatrix} 1 & 0 & 1 & 0 \\ 0 & 1 & 0 & 1 \end{pmatrix} \in \R^{2\times 4}
\end{equation*}
and
\begin{equation*}
  Z_s := \frac{X_s^1-X_s^2}{|X_s^1-X_s^2|} \in \R^2,
\end{equation*}
we obtain from \cite[Theorem II.29]{Protter04}
that 
\comm{[Theorems II.29 p.75 for simpler processes and IV.12 p.159 for more general ones. T II.29 p.75 needs the $\Sigma$ processes to be caglad and adapted, which they are ($X^i_s$ is even continuous by our contradiction-assumption)]}
\begin{equation*}
  [\beta,\widetilde W]_t 
  = \int_0^t \Sigma^\beta_s (\Sigma^W)^\top\, ds
  = \frac12 \int_0^t  (Z_s^\top - Z_s^\top) \, ds
  = 0
\end{equation*}
for all $t \in [0,T]$.
A similar, simpler computation shows that also $[\beta, W^i]_t = 0$ for all $t \in [0,T]$ and all $3 \leq i \leq N$. This proves the claim.
\smallskip

\textbf{Step 4:} Reaching a contradiction with~\eqref{pfzq}. We have by the equivalence between $\P$ and $\P_*$ that \eqref{pfzq} is equivalent to the statement 
\[
  \P_*\Big( \inf_{0 \leq t \leq T} D_t = 0 \Big) =0.
\]
By reducing the time interval and taking $i=1$ and $j=2$ in the definition of $D_t$ in \eqref{Dt}, we obtain
\[
  \Big\{ \inf_{0 \leq t \leq T} D_t = 0 \Big\} 
  \supset \Big\{ \inf_{0 \leq t \leq t_0} |X_t^1 - X_t^2| = 0 \Big\}
  =: A
\]
for any $t_0 \in (0,T)$. If \eqref{pfzp} would hold with $\tau$ replaced by any $t_0 > 0$, then $|X_t^1 - X_t^2|^2 = 2 \sigma^2 R_t$, and Theorem \ref{t:BP}\ref{t:BP:less2} on Bessel processes (with dimension $\delta = 2(1-\sigma^{-2})<2$) would imply $\P_* (A) > 0$. Then, the contradiction would be reached. Based on this idea, we are going to show that
\begin{equation} \label{pfza}
  \P_* (A \cap \{\tau > t_0\}) > 0
\end{equation}
for $t_0 > 0$ small enough. To prove this we extend in Step 4a the systems \eqref{pfzy}, \eqref{pfyx} and \eqref{pfyy} in time;
we denote the resulting processes respectively by $\tilde X_t^i$ for $3 \leq i \leq N$, $\tilde S$ and $\tilde R$. We then construct three events $\Omega_1$, $\Omega_2$ and  $\Omega_3$ (Step 4b) with the following properties (Steps 4b-4d) when $t_0 > 0$ is small enough:
\begin{enumerate}[label=(\roman*)]
   \item \label{pfyz1} $\Omega_1$ is defined in terms of $\tilde R$, $\Omega_2$ in terms of $\tilde S$ and $\Omega_3$ in terms of $\tilde X^i$  for $3 \leq i \leq N$,
   \item \label{pfyz2} $\P_* (\Omega_k) > 0$ for all $1 \leq k \leq 3$,
   \item \label{pfyz3} The collection $\{\Omega_1,\Omega_2,\Omega_3\}$ is $\P_*$-independent,
   \item \label{pfyz4} For any $1 \leq k \leq 3$, we have on $\Omega_k$ that $\tau_k > t_0 \wedge \tau$ (consequently, $\cap_{k=1}^3 \Omega_k \subset \{ \tau > t_0 \}$), and
   \item \label{pfyz5} $\cap_{k=1}^3 \Omega_k \subset A$.
 \end{enumerate} 
Then, \eqref{pfza} follows from
\[
  \P_* (A \cap \{\tau > t_0\}) 
  \geq \P_* \Big( \bigcap_{k=1}^3 \Omega_k \Big)
  = \prod_{k=1}^3 \P_* ( \Omega_k)
  > 0,
\]
and the contradiction with \eqref{pfzq} is reached.
\smallskip 

\textbf{Step 4a}: The extension in time of \eqref{pfzy}, \eqref{pfyx} and \eqref{pfyy}. Note that $\bX_t$, $\bW_t$, and $\beta_t$ are defined on $[0,T]$. Then, we extend \eqref{pfyx} in time by defining the process
\begin{equation*}
  \tilde S_t = 2 \ol x_0 + \sigma (W_t^1 + W_t^2) \qquad \text{for all } 0 \leq t \leq T.
\end{equation*}
This is an extension of $S_t$ in time in the sense that $\tilde S_t = S_t = X_t^1 + X_t^2$ for $t \in [0,\tau]$. Similarly, we extend \eqref{pfyy} in time by considering
\begin{equation} \label{pfyw}
\left\{ \begin{aligned}
d \tilde R_t &= 2 \sqrt{\tilde R_t} d \beta_t + 2\Big(1 - \frac1{\sigma^2} \Big) \, dt 
    &&\text{on } (0, T] \\
    \tilde R_0 &= R_0 = \frac{|x_1^0 - x_2^0|^2}{2 \sigma^2}.
    &&
\end{aligned} \right.
\end{equation}
By Lemma~\ref{l:sqBessel},  \eqref{pfyw} has a unique strong solution $\tilde R_t$ (stopped at zero, if necessary), which therefore coincides with $R_t$ as long as $R_t >0$, and therefore coincides with $R_t$ for all $t\in [0,\tau]$.

Finally, considering \eqref{pfyy}, the system 
\comm{[equality of strong solutions is per KaratzasShreve's definition of strong solution.]}
\begin{equation} \label{pfzy:ext}
  \left\{ \begin{aligned}
    d \tilde X_t^i 
    &= \sigma dW_t^i + \sum_{ j = 3 }^N b^i b^j K(\tilde X_t^i - \tilde X_t^j) dt 
    && \text{on $(0, \tilde \tau)$ for all } 3 \leq i \leq N \\
    \tilde X_0^i 
    & = X_0^i = x_0^i
    && \text{for all } 3 \leq i \leq N.
  \end{aligned} \right.  
\end{equation}
attains a unique strong solution up to the first collision time capped at $T$ given by
\[
  \tilde \tau := T \wedge \lim_{\ell \to \infty} \inf \Big\{ t \in [0,T] : \min_{3 \leq i < j \leq N} |\tilde X_t^i - \tilde X_t^j| \leq \frac1\ell \Big\}.
\]
Note that $(\tilde X_t^3, \ldots, \tilde X_t^N)$ coincides with $(X_t^3, \ldots, X_t^N)$ on $[0,\tau \wedge \tilde \tau)$. In fact, it can be shown  that $\tilde \tau > \tau$, but we do not need this in the sequel. 
\comm{[Proof of "$\tilde \tau > \tau$": $(X_t^3, \ldots, X_t^N)$ satisfies \eqref{pfzy:ext} on $[0,\tau]$, by the supposed \eqref{pfzq} there are no collisions on $[0,\tau]$, and thus the uniqueness of solutions until the first collisions implies that $(\tilde X_t^3, \ldots, \tilde X_t^N)$ also has to satisfy \eqref{pfzy:ext} until $[0,\tau]$ without collisions a.s.]}

\smallskip

\textbf{Step 4b}: definitions of $\Omega_k$ and Properties \ref{pfyz1}, \ref{pfyz3} and \ref{pfyz5}. We take
\begin{align*}
  \Omega_1 &:= \Big\{ \min_{0 \leq t \leq t_0} \tilde R_t = 0, \ 
                     \max_{0 \leq t \leq t_0} \sqrt{2 \sigma^2 \tilde R_t } < \Big( a + \frac12 \Big) \Dopp_0 \Big\}, \\
  \Omega_2 &:= \Big\{ \max_{0 \leq t \leq t_0} | { \tilde S_t } - 2 \ol x_0 | < \Big( a - \frac12 \Big) \Dopp_0 \Big\}, \\
  \Omega_3 &:= \Big\{ \tilde \tau > t_0, \ \min_{0 \leq t \leq t_0} \min_{3 \leq j \leq N} |\tilde X_t^j - \ol x_0| > b \Dopp_0 \Big\}.
\end{align*}
Note that for each $\Omega_k$ the estimates on $\tilde R$, $\tilde S$ and $\tilde X_i$ are the negation of the estimate in the definition of the corresponding $\tau_k$, {but defined in terms of the extended processes instead of $\bX$.}

Turning to the properties \ref{pfyz1}--\ref{pfyz5} above, first note that property~\ref{pfyz1} is satisfied by construction. Property \ref{pfyz4} implies  \ref{pfyz5}, because if \ref{pfyz4} holds, then on  $\cap_{k=1}^3 \Omega_k$ we have $t_0 < \tau$, hence $\tilde R_t = R_t$ on $[0,t_0]$, and thus $\Omega_1 \subset A$. 
Finally, \ref{pfyz3} follows from \ref{pfyz1} and the fact that the processes  $\tilde R$, $\tilde S$ and $\tilde X_i$ are $\P_*$-independent on $[0,T]$. We next continue with properties~\ref{pfyz2} and~\ref{pfyz4}. 

\textbf{Step 4c}: Property \ref{pfyz2}. By Theorem~\ref{t:BP}\ref{t:BP:less2} we obtain $\P_*(\Omega_1) > 0$. For $\Omega_2$ we simply observe that $\tilde S_t - 2 \ol x_0 = \sigma (W_t^1 + W_t^2)$ is under $\P_*$ a $2$-dimensional Brownian motion multiplied by $\sqrt 2 \sigma$. Hence, $\P_*(\Omega_2) > 0$.

For $\Omega_3$ we note that since it is the intersection of two events, it is sufficient to show that for $t_0 > 0$ small enough both events have $\P_*$-probability larger than $\frac12$. For both events we rely on the path-continuity of the solution to \eqref{pfzy:ext}. For the first event $\{\tilde \tau > t_0\}$, we have by $D_0 > 0$ that $\P_*(\tilde \tau > 0) = 1$, and thus for $t_0$ small enough we have  $\P_*(\tilde \tau > t_0) > \frac12$. Similarly, for the second event, we recall from \eqref{pfzu} and \eqref{pfzt} that
\[
  \min_{3 \leq j \leq N} |x_0^j - \ol x_0| > b \Dopp_0.
\] 
Hence, for $t_0$ small enough 
\[
  \P_* \Big( \min_{0 \leq t \leq t_0} \min_{3 \leq j \leq N} |\tilde X_t^j - \ol x_0| > b \Dopp_0 \Big) > \frac12.
\]

\textbf{Step 4d}: Property \ref{pfyz4}. We rely on the overlap of the extended processes $\tilde R$, $\tilde S$ and $\tilde X_i$ with the original processes $ R$, $ S$ and $ X_i$ on $[0,\tau]$, and use a similar argument for each $\Omega_k$. On $\Omega_1$, we have
\[
  \Big( a + \frac12 \Big) \Dopp_0
  > \max_{0 \leq t \leq t_0} \sqrt{2 \sigma^2 \tilde R_t } 
  \geq \max_{0 \leq t \leq t_0 \wedge \tau} \sqrt{2 \sigma^2 \tilde R_t } 
  = \max_{0 \leq t \leq t_0 \wedge \tau} \sqrt{2 \sigma^2 R_t }
  = \max_{0 \leq t \leq t_0 \wedge \tau} | X_t^1 -  X_t^2|.
\]
Hence, by the definition of $\tau_1$, we obtain on $\Omega_1$ that $\tau_1 > t_0 \wedge \tau$. 
On $\Omega_2$, we have similarly
\[
  \Big( a - \frac12 \Big) \Dopp_0
  > \max_{0 \leq t \leq t_0 \wedge \tau} | \tilde S_t - 2 \ol x_0|
  = \max_{0 \leq t \leq t_0 \wedge \tau} | X_t^1 + X_t^2 - 2 \ol x_0|
\]
and thus $\tau_2 > t_0 \wedge \tau$. On $\Omega_3$, we again have similarly
\[
  b \Dopp_0
  < 
  \min_{0 \leq t \leq t_0 \wedge \tau} \min_{3 \leq j \leq N} |X_t^j - \ol x_0| 
\]
and thus $\tau_3 > t_0 \wedge \tau$. 
\end{proof}

\begin{rem} \label{r:coll:pf}
We briefly comment on the choice in Step 4a to extend the processes $S_t$ and $R_t$ in time as opposed to the more natural choice to extend instead the system for $(X_t^1, X_t^2)$ in \eqref{pfzp}. The reason is that in the time-extended version of \eqref{pfzp} collisions are possible under $\P_*$, and then we do not have a uniqueness concept for the solution to \eqref{pfzp} (see Section \ref{s:intro:N2} for an alternative representation). These issues can conveniently be avoided by working instead with the processes $S_t$ and $R_t$.
\end{rem}

\newpage

\section{Global existence}
\label{s:global-existence}

In this section we state and prove Theorem \ref{t:ex} on the global existence of weak solutions $\bX$ to \eqref{SDE:full:param}. For such $\bX$ (if it exists), let 
\begin{equation} \label{sD}
    \Dsame_t := \min_{\substack{ i \neq j \\ b^i = b^j}} |X_t^i - X_t^j|
\end{equation}
be the shortest distance between any two particles with equal sign. 

\begin{thm}[Existence and properties of weak solutions] \label{t:ex}
Let $N \geq 2$ and $\sigma, \gamma > 0$ be such that $\gamma < \frac12 \sigma^2 $. Let $\bb \in \{-1,1\}^N$ and let $N_+$ and $N_-$ be the numbers of positive and negative $b^i$. Let the random initial condition $\bX_0$ be exchangeable with respect to $\bb$ and such that $\P( \exists \, i \neq j : X_0^i = X_0^j) = 0$ and $\E [|\bX_0|] < \infty$. Then:
\begin{enumerate}
\item \label{t:ex:fixed-signs}
there exists a global weak solution $\bX$ to \eqref{SDE:full:param}. Moreover, $\bX_t$ is exchangeable with respect to $\bb$, $\E [|\bX_t|] < \infty$ for all $t > 0$,
\begin{equation} \label{pfzs}  
  \int_0^t |X_s^i - X_s^j|^{\alpha - 2} ds < \infty
  \qquad \text{a.s.\ for all } t \in [0, \infty), \text{ all } 1 \leq i < j \leq N \text{ and all } \frac{2 \gamma}{\sigma^2} < \alpha \leq 3,
\end{equation}
and
\begin{equation} \label{pfzb}
  \P \Big( \inf_{0 \leq t \leq T} \Dsame_t = 0 \Big) = 0
  \qquad \text{for all } T > 0.
\end{equation}
\item\label{t:ex:stronger-estimate} if, moreover, $\gamma < \sigma^2/(2(N-1))$, then for any 
\[
 \frac{2\gamma (N-1)}{\sigma^2 } < \alpha <1
\]
we have the additional estimate
\begin{equation}
  \label{c:ex:est}
  \frac1{N(N-1)} \sum_{\substack{i,j=1\\i\not=j}}^N \E \int_0^T |X_t^i - X_t^j|^{\alpha - 2} dt
  \leq
  \frac{2 \sqrt 2}{\alpha\delta\beta_0} \biggl(1 + \frac54 \sigma^2 T + \frac1N \sum_{i=1}^N \E |X_0^i|\biggr)
  \end{equation}
  
  for any $T > 0$ and any $\beta_0>0$ such that  $\beta_0 N\leq N_-, N_+$, where
\[
  \delta := \alpha \sigma^2 - 2\gamma(N-1) > 0.
\]
\end{enumerate}
\end{thm}

\begin{rem} The estimate \eqref{pfzs} motivates the requirement $\gamma < \frac12 \sigma^2$. Indeed, for the drift in the system \eqref{SDE:full:param} to be integrable, it is sufficient to have \eqref{pfzs} with $\alpha = 1$. This requires $\gamma < \frac12 \sigma^2$. The strict inequality gives some `wiggle room' for taking $\alpha$ slightly smaller than $1$, which we exploit in the proofs of Theorems \ref{t:ex} and \ref{t:mf-limit}.  
\end{rem}

We give the proof of Theorem \ref{t:ex} at the end of this section. We start with preparations. Without loss of generality we may assume that
\begin{equation*}
  N_+ \geq N_- \geq 1.
\end{equation*}
Indeed, Proposition \ref{p:obv}\ref{p:obv:bflip} demonstrates that the first inequality is not restrictive, and in the single-sign case $N_- = 0$ the proof becomes easier (we omit the details).

We obtain a weak solution to \eqref{SDE:full:param} by constructing a process $\bX^{\e, \ell}$ and taking the limit as $\e \to 0$ and $\ell \to \infty$. For $\e \in (0,1)$ let
\begin{equation}
  \label{eqdef:K_e}
  K_\e (x) := \dfrac x{|x|^2 + \e^2}
\end{equation}
be a regularized interaction force. For any integer $\ell \geq 1$, we are going to define a cut-off function $\Phi_\ell : (\R^2)^N \to [0,1]$ which satisfies
\begin{equation} 
  \label{eq:prop:Phi_ell}
  \Phi_\ell (\bX_t) = \begin{cases}
    1
    &\text{if } \Dsame_t \geq \frac1\ell \\
    0
    &\text{if } \Dsame_t \leq \frac1{2\ell}.
  \end{cases}
\end{equation}
With this aim, take $\ell \geq 1$ and let
\begin{equation*}  
  \cU_\ell := \Big\{ \bx \in (\R^2)^N : \min_{i \neq j, \, b^i = b^j} |x_i - x_j| > \frac1\ell \Big\}.
\end{equation*}
Clearly, $\cU_\ell$ is increasing as a set in $\ell$. In fact, $\dist (\cU_k, \cU_\ell^c) > 0$ is increasing in $\ell $ for $\ell > k$. Hence, there exists a Lipschitz continuous cut-off function $\Phi_\ell : (\R^2)^N \to [0,1]$ which satisfies
\begin{equation*} 
  \Phi_\ell (\bx) = \begin{cases}
    1
    &\text{if } \bx \in  \cU_\ell \\
    0
    &\text{if } \bx \in  \cU_{2 \ell}^c.
  \end{cases}
\end{equation*}
Hence, $\Phi_\ell$ satisfies~\eqref{eq:prop:Phi_ell}.

Using $K_\e$ and $\Phi_\ell$, we define the $\e, \ell$-regularized system as
  \begin{equation} \label{SDE:e:ell}
  d X_t^{i,\e,\ell} = \sigma dB_t^i + \gamma \Phi_\ell (\bX_t^{\e,\ell}) \sum_{j=1}^N b^i b^j K_\e (X_t^{i,\e,\ell} - X_t^{j,\e,\ell}) dt,
\end{equation}
with the same initial condition $\bX_0$ as for $\bX$.
The role of $K_\e$ is to ensures existence and uniqueness of a strong solution $\bX^{\e, \ell}$. The role of $\Phi_\ell$ is to gradually turn off the drift for \textit{each} particle whenever two particles of the same sign come close together. Note that this also automatically turns off the drift if three or more particles come close together, regardless of their signs, as then there have to be at least two particles with the same sign.  

We start by proving the following adjusted version of \cite[Lemma 12]{FournierJourdain17}:

\begin{lem} [Uniform bounds on $\bX^{\e, \ell}$] \label{l:ex:apri}
Let the setting in Theorem \ref{t:ex} be given and assume $N_+ \geq N_- \geq 1$.
For any $\ell \geq 1$ and any $\e \in (0,1)$, let $\bX^{\e, \ell}$ be the solution of \eqref{SDE:e:ell} with $\bX_0^{\e, \ell} = \bX_0$. Then, $\bX_t^{\e, \ell}$ is exchangeable with respect to $\bb$ and the following three properties hold:
\begin{enumerate}[label=(\roman*)]
  \item \label{l:ex:apri:E} Let $\phi(x) := \sqrt{1 + |x|^2}$. For all $1 \leq i \leq N$ and all $t \geq 0$, 
  \[
    \E \phi(X_t^{i,\e,\ell})
    \leq \frac1{N_-} \sum_{j=1}^N \E \phi(X_0^j) + \frac{N}{N_-} \Big( \sigma^2 + \frac {\gamma (N-1)}2 \Big) t; 
  \]
  \item \label{l:ex:apri:al}
  Let $T>0$ and $2 \gamma \sigma^{-2} < \alpha  \leq 3$. Then there exists a constant $C>0$ independent of $\e$ such that for all $1 \leq i < j \leq N$
  \begin{align*} 
  \int_0^T \E \big[ |X_t^{i,\e,\ell} - X_t^{j,\e,\ell}|^{ \alpha - 2} \big] \, dt 
  &\leq C.
  \end{align*} 
  \item \label{l:ex:apri:stronger}
  In the setting of part~\ref{t:ex:stronger-estimate} of Theorem~\ref{t:ex}, $\bX^{\e,\ell}$ satisfies the upper bound~\eqref{c:ex:est}. 
\end{enumerate}
\end{lem}

We remark that \ref{l:ex:apri:E} also holds for $\gamma \geq \frac12 \sigma^2$, 
and that the constant $C$ in \ref{l:ex:apri:al} is explicit (see~\eqref{pfzk}) but not optimal.

\begin{proof}[Proof of Lemma \ref{l:ex:apri}] 
For simplicity we set $\bX := \bX^{\e, \ell}$ (we do not use \eqref{SDE:full:param} in the proof).
The exchangeability with respect to $\bb$ follows directly from the system \eqref{SDE:e:ell}. 
\smallskip

\textbf{Proof of \ref{l:ex:apri:E}}. By Proposition \ref{p:obv} we may assume that $i=1$, $b^2 = 1$ and $b^3 = -1$. We reuse $i$ as a generic index in the remainder.

 Applying It\^o's lemma to $\sum_{i=1}^N \phi(X_t^i)$ and taking the expectation, we obtain
\begin{multline} \label{pfzo}
    \sum_{i=1}^N \E \phi(X_t^i) 
   =  \sum_{i=1}^N \E \phi(X_0^i) 
     +  \gamma \sum_{i=1}^N \int_0^t \E \Big( \Phi_\ell (\bX_s) \sum_{j=1}^N b^i b^j K_\e (X_s^i - X_s^j) \cdot \nabla \phi (X_s^i) \Big) \, dt \\
     +  \sum_{i=1}^N \frac{\sigma^2}2 \int_0^t \E \Delta \phi(X_s^i) \, dt.
\end{multline} 

We treat the four terms in \eqref{pfzo} separately. The term concerning $\bX_0$ needs no adjustment; it appears almost directly in the desired estimate in \ref{l:ex:apri:E}. 
For the term concerning $\bX_t$, we obtain from the exchangeability with respect to $\bb$ (assuming $N \geq 3$ for convenience)
\[
  \sum_{i=1}^N \E \phi(X_t^i)
  = N_+ \E \phi(X_t^2) + N_- \E \phi(X_t^3)
  \geq N_- \E \phi(X_t^1).
\] 

For the remaining two terms in \eqref{pfzo}, we start with some preparations. For any $x = (x_1, x_2) \in \R^2$ we have
\[
  \nabla \phi (x) = \frac x{\phi (x)} \qquad\text{and}\qquad
  \nabla^2 \phi (x) = \frac1{\phi(x)^3} \begin{bmatrix} 1 + x_2^2 & -x_1 x_2 \\  -x_1 x_2 & 1 + x_1^2 \end{bmatrix}.
\]
The eigenvalues of $\nabla^2 \phi$ are $\frac1\phi, \frac1{\phi^3} \in (0,1]$, and therefore we have 
\[
|\nabla\phi(x)-\nabla\phi(x')|\leq |x-x'|
\qquad\text{and}\qquad  
\Delta \phi = \frac1\phi + \frac1{\phi^3} \leq 2. 
\]
Using the latter inequality we estimate the last term in \eqref{pfzo} as
\[
  \sum_{i=1}^N \frac{\sigma^2}2 \int_0^t \E \Delta \phi(X_s^i) \, dt
  \leq \sigma^2 N t.
\]
The term involving $K_\e$ requires several steps. First, we use the antisymmetry of $K_\e$ to rewrite the sum over $j$ as
\begin{equation} \label{pfzn}  
  \sum_{j=1}^N b^i b^j K_\e (X_s^i - X_s^j) \cdot \nabla \phi (X_s^i)
  = \frac12 \sum_{\substack{j =1\\j\not=i}}^N b^i b^j K_\e (X_s^i - X_s^j) \cdot \big( \nabla \phi (X_s^i) - \nabla \phi (X_s^j) \big).
\end{equation}
Second, we estimate the summand in absolute value. Component-wise, we apply $|b^i b^j| = 1$, $| K_\e (X_s^i - X_s^j) | \leq |X_s^i - X_s^j|^{-1}$, and $|\nabla \phi (X_s^i) - \nabla \phi (X_s^j)| \leq |X_s^i - X_s^j|$. Hence, the sum in \eqref{pfzn} is bounded from above by $\frac12 (N-1)$. Third, using that $|\Phi_\ell| \leq 1$, we obtain
\[
  \gamma \sum_{i=1}^N \int_0^t \E \Big( \Phi_\ell (\bX_s) \sum_{j=1}^N b^i b^j K_\e (X_s^i - X_s^j) \cdot \nabla \phi (X_s^i) \Big) \, dt
  \leq \frac{\gamma N(N-1)}2 t.
\]
In conclusion, by substituting the obtained estimates in \eqref{pfzo}, we obtain (i).
\smallskip

\textbf{Proof of \ref{l:ex:apri:al}}. To cover all four cases for the possible values of $b^i$ and $b^j$, we relabel the indices such that $b^1 = b^2 = 1 = - b^{N-1} = -b^N$. In the special case where $N_- = 1$, the proof below applies with $b^{N-1} = 1$. If, moreover, $N=2$, the proof applies with $b^2 = -1$.
We split three cases for $\alpha$. If $2 \leq \alpha \leq 3$, then \ref{l:ex:apri:al} follows from \ref{l:ex:apri:E} by
\[
  |X_t^i - X_t^j|^{ \alpha - 2}
  \leq 1 + |X_t^i - X_t^j|
  \leq 1 + |X_t^i| + |X_t^j|.
\] 
We treat the case $1 \leq \alpha < 2$ at the end of the proof as a corollary of the remaining case $2 \sigma^{-2} < \alpha < 1$, which we consider from here onwards.

\medskip

By the Monotone Convergence Theorem, it is sufficient to show that
\begin{align} \label{pfzi} 
  A_{ij} := \int_0^T \E \Big[ \big(|X_t^{i} - X_t^{j}|^2 + \eta^2 \big)^{\frac\alpha2 - 1} \Big] \, dt
  &\leq C
\end{align} 
for all $\eta \in (0,\e]$. Let 
\[
  \phi_\eta : \R^2 \to [0,\infty), \quad 
  \phi_\eta (x) := \sqrt{ |x|^2 + \eta^2 },
\]
and note that $\phi_1 = \phi$. We are going to apply It\^o's lemma to 
\begin{equation*}
  f(\bX) := \sum_{i=1}^N \sum_{j \neq i} \phi_\eta^\alpha(X^i - X^j).
\end{equation*}
As preparation for this, we compute (using that $\phi_\eta$ is even)
\begin{align*}
  \nabla_{X^i} f(\bX) &= 2 \sum_{j \neq i} \nabla \phi_\eta^\alpha(X^i - X^j),  \\
  \Delta f(\bX) &= 2 \sum_{i=1}^N \sum_{j \neq i} \Delta \phi_\eta^\alpha(X^i - X^j).
\end{align*}
Similar to the proof of \ref{l:ex:apri:E}, we apply It\^o's lemma and take the expectation. This  yields
\begin{subequations} \label{pfzm} 
\begin{align} \label{pfzm1}
   &\sum_{i=1}^N \sum_{j \neq i} \E \phi_\eta^\alpha(X_T^i - X_T^j) \\\label{pfzm2}
   &=  \sum_{i=1}^N \sum_{j \neq i} \E \phi_\eta^\alpha(X_0^i - X_0^j) \\\label{pfzm3}
   &\quad  + 2\gamma  \sum_{i=1}^N \sum_{j \neq i} \int_0^T \E \Big[ \Phi_\ell (\bX_t) \sum_{k \neq i} b^i b^k K_\e(X_t^i - X_t^k) \cdot \nabla \phi_\eta^\alpha(X_t^i - X_t^j) \Big] \, dt \\\label{pfzm4}
   &\quad  + \sigma^2 \sum_{i=1}^N \sum_{j \neq i} \int_0^T \E \Delta \phi_\eta^\alpha(X_t^i - X_t^j) \, dt.
\end{align} 
\end{subequations}

To obtain \eqref{pfzi} we estimate the four terms in \eqref{pfzm} separately. We will bound the term \eqref{pfzm1} from above by a constant. The term \eqref{pfzm2} is nonnegative; we simply bound it from below by $0$. We will split the drift term \eqref{pfzm3} into two parts which we both bound from below; one by $-C' \sum_{ij} A_{ij}$ and the other one by $-C \sum_{ij} A_{ij}^{1/(2-\alpha)}$. Finally, we will bound \eqref{pfzm4} from below by the \emph{positive} value $C'' \sum_{ij} A_{ij}$ with $C'' > C'$. To obtain this ordering of the constants we rely on $\gamma < \frac12 \sigma^2$. Then, applying all these bounds to \eqref{pfzm}, we will derive \eqref{pfzi}.

To estimate \eqref{pfzm4} from below, we use $\Delta \phi_\eta^\alpha (x) \geq \alpha^2 \phi_\eta^{\alpha-2} (x)$ to obtain
\begin{equation}
  \label{est:Ito-term}
  \sigma^2 \sum_{i=1}^N \sum_{j \neq i} \int_0^T \E \Delta \phi_\eta^\alpha(X_t^i - X_t^j) \, dt
  \geq \sigma^2 \alpha^2 \sum_{i=1}^N \sum_{j \neq i} \int_0^T \E \phi_\eta^{\alpha-2} (X_t^i - X_t^j)  \, dt 
  = \sigma^2 \alpha^2 \sum_{i=1}^N \sum_{j \neq i} A_{ij}.
\end{equation}
Note from the exchangeability with respect to $\bb$ that 
\begin{align*}
  \sum_{i=1}^N \sum_{j \neq i} A_{ij}
  = N_+ (N_+ - 1) A_{12} + 2 N_+ N_- A_{1N} + N_- (N_- - 1) A_{N-1,N}
  =: \sum_{m=1}^3 N_m A_m. 
\end{align*}

To estimate \eqref{pfzm1} from above, we rely on \ref{l:ex:apri:E} and on
\[
  \phi_\eta^\alpha (x - y) 
  \leq \phi (x - y) 
  \leq \sqrt 2 (\phi(x) + \phi(y)), 
\]
where we recall that $\phi(x) = \sqrt{|x|^2 + 1} \geq 1$.
Then
\begin{align} \notag 
  \sum_{i=1}^N \sum_{j \neq i} \E \phi_\eta^\alpha(X_T^i - X_T^j) 
  &\leq 2 \sqrt 2 (N-1) \sum_{i=1}^N \E \phi(X_T^i) \\\label{pfzl}
  &\leq  2 \sqrt 2 N(N-1) \bigg( \frac1{N_-} \sum_{i=1}^N \E \phi(X_0^i) + \frac{ N}{N_-} \Big( \sigma^2 + \frac {\gamma (N-1)}2 \Big) T \bigg) =: C_0.
\end{align}

Finally, we estimate \eqref{pfzm3} in absolute value. By moving the absolute value inside the sum over $k$ and then applying $|b^i b^j| = 1$, $|K_\e(x)| \leq \phi_\e^{-1}(x) \leq \phi_\eta^{-1}(x)$ and $|\nabla \phi_\eta^\alpha (x)| \leq \alpha \phi_\eta^{\alpha -1}(x)$, we get
\begin{multline}
  \label{est:drift-term}
  \bigg| 2\gamma  \sum_{i=1}^N \sum_{j \neq i} \int_0^T \E \Big[ \Phi_\ell (\bX_t) \sum_{k \neq i} b^i b^k K_\e(X_t^i - X_t^k) \cdot \nabla \phi_\eta^\alpha(X_t^i - X_t^j) \Big] \, dt \bigg| \\
  \leq 2 \gamma \alpha  \sum_{i=1}^N \sum_{j \neq i} \sum_{k \neq i} \int_0^T \E \big[ |\Phi_\ell (\bX_t)| \phi_\eta^{\alpha -1}(X_t^i - X_t^j) \phi_\eta^{-1}(X_t^i - X_t^k) \big] \, dt.
\end{multline}
Next, we split off the diagonal $k=j$ from the summation. For this diagonal contribution, we apply $|\Phi_\ell| \leq 1$ and get
\begin{align}
  \label{est:diagonal-term}
  2\gamma  \alpha \sum_{i=1}^N \sum_{j \neq i} \int_0^T \E \big[ |\Phi_\ell (\bX_t)| \phi_\eta^{\alpha -1}(X_t^i - X_t^j) \phi_\eta^{-1}(X_t^i - X_t^j) \big] \, dt
  \leq 2 \gamma \alpha \sum_{i=1}^N \sum_{j \neq i} A_{ij}.
\end{align}
For the off-diagonal contribution $k \neq j$, we estimate the integrand by applying the following steps:
\begin{itemize}
   \item Use
   \[ u^{\alpha-1} v^{-1} 
   \leq \Big( 1 + \frac1u \Big) \frac1v 
   = \Big( 1 + \frac1{u+v} \Big) \frac1v + \frac1{(u+v)u}.
   \]
   \item Note that
   \[
     \phi_\eta(x_i - x_j) + \phi_\eta(x_i - x_k)
     \geq |x_i - x_j| + |x_i - x_k|.
   \]
   We claim that if $\Phi_\ell (\bx) \neq 0$, then
   \begin{align} \label{pfzc}
     |x_i - x_j| + |x_i - x_k| \geq \frac1{2\ell}.
   \end{align}
   Indeed, out of the three signs $b^i, b^j, b^k$, at least two are equal. Then, $\Phi_\ell (\bx) \neq 0$ implies $\max \{|x_i - x_j|, |x_i - x_k|,|x_j - x_k| \} \geq 1/(2 \ell)$. Depending on which of the three distances is largest, \eqref{pfzc} follows either directly or from the triangle inequality.
   
   \item Use $|\Phi_\ell| \leq 1$.
\end{itemize} 
This yields
\begin{align*}
  |\Phi_\ell(\bx)| \phi_\eta^{\alpha -1}(X_t^i - X_t^j) \phi_\eta^{-1}(X_t^i - X_t^k)
  &\leq \frac{2 \ell + 1}{ \phi_\eta(X_t^i - X_t^k) } + \frac{2 \ell}{ \phi_\eta(X_t^i - X_t^j)}.
\end{align*}
To use this for estimating the off-diagonal, we 
\begin{itemize}
   \item apply H\"older to obtain
   \[
     \E \int_0^T \phi_\eta^{-1} (Y_t) \, dt
     \leq T^{\tfrac{1-\alpha}{2-\alpha}} \Big( \E \int_0^T \phi_\eta^{\alpha - 2} (Y_t) \, dt \Big)^{\tfrac1{2-\alpha}};
   \]
   \item use the exchangeability of $\bX_t$ with respect to $\bb$.
 \end{itemize} 
Then, we estimate the off-diagonal contribution as
\begin{align} 
  \notag 
  \MoveEqLeft 2 \gamma \alpha \sum_{i=1}^N \sum_{j \neq i} \sum_{k \neq i,j} \int_0^T \E \big[ |\Phi_\ell (\bX_t)| \phi_\eta^{\alpha -1}(X_t^i - X_t^j) \phi_\eta^{-1}(X_t^i - X_t^k) \big] \, dt \notag \\
  &\leq 2\gamma  \alpha N (4 \ell + 1) \sum_{i=1}^N \sum_{j \neq i} \int_0^T \E \phi_\eta^{-1} (X_t^i - X_t^j) \, dt\notag  \\
  &\leq 10 \gamma \alpha \ell N T^{\tfrac{1-\alpha}{2-\alpha}} \sum_{i=1}^N \sum_{j \neq i}  \Big( \underbrace{ \int_0^T \E \phi_\eta^{\alpha - 2} (X_t^i - X_t^j) \, dt }_{A_{ij}} \Big)^{\tfrac1{2-\alpha}} \notag \\
  &= 10\gamma  \alpha \ell N T^{\tfrac{1-\alpha}{2-\alpha}} \sum_{m=1}^3 N_m A_m^{\tfrac1{2-\alpha}}.
  \label{est:off-diagonal-terms}
\end{align}

Finally, combining the estimates~\eqref{est:Ito-term}, \eqref{pfzl}, \eqref{est:diagonal-term}, and~\eqref{est:off-diagonal-terms} with~\eqref{pfzm}, we get 
\begin{align} \label{pfzj}
  C_1 \sum_{m=1}^3 N_m A_m 
  \leq C_0 + C_2 \ell \sum_{m=1}^3 N_m A_m^{\tfrac1{2-\alpha}},  
\end{align}
where $C_0$ is defined in \eqref{pfzl} and
\[
  C_1 := \alpha \Big( \alpha \sigma^2 - 2\gamma  \Big) > 0,
  \qquad C_2 := 10 \gamma \alpha N T^{\tfrac{1-\alpha}{2-\alpha}}.
\] 
To obtain the desired inequality \eqref{pfzi}, let $A_* = \max_{1\leq m \leq 3} A_m$ and assume $A_* \geq 1$ (otherwise \eqref{pfzi} is trivial). Using that $\frac12 N_-^2 \leq N_m \leq 2 N_+^2$, we obtain from \eqref{pfzj}
\[
  A_*
  \leq \frac{2 C_0}{N_-^2 C_1} + \frac{12 N_+^2 C_2 \ell}{N_-^2 C_1} A_*^{\tfrac1{2-\alpha}}
  \leq C_3 A_*^{\tfrac1{2-\alpha}}, 
  \qquad C_3 := 2\frac{C_0 + 6 N_+^2 C_2 \ell}{N_-^2 C_1}.
\] 
Hence,
\begin{align} \label{pfzk}
  A_m \leq A_* \leq C_3^{\tfrac{2-\alpha}{1-\alpha}} =: C_4.
\end{align}

Finally, we treat the case $\alpha \in [1,2)$. Take $\alpha' \in (2 \sigma^{-2}, 1)$ (e.g.\ as the midpoint), and let $C_4'$ be as in \eqref{pfzk} with respect to $\alpha'$. Then,
\[
  \int_0^T \E \big[ |X_t^{i} - X_t^{j}|^{ \alpha - 2} \big] \, dt 
  \leq
  \int_0^T \E \big[ |X_t^{i} - X_t^{j}|^{ \alpha' - 2} + 1 \big] \, dt 
  \leq C_4' + T.
\]
\smallskip

\textbf{Proof of \ref{l:ex:apri:stronger}}. 
We follow the proof of \ref{l:ex:apri:al} until \eqref{est:drift-term} and use rougher estimates onwards. Starting from the right-hand side of~\eqref{est:drift-term}, we use $|\Phi_\ell|\leq 1$, Young's inequality and exchangeability
to find 
\begin{align*}
\MoveEqLeft{2\gamma\alpha \sum_{i=1}^N \sum_{j \neq i} \sum_{k \neq i} \int_0^T \E \bigl[\phi_\eta^{\alpha -1}(X_t^i - X_t^j) \phi_\eta^{-1}(X_t^i - X_t^k) \bigr] \, dt}\\
&\leq 2\gamma\alpha \sum_{i=1}^N \sum_{j \neq i} \sum_{k \neq i} \int_0^T \E \Bigl[\frac{1-\alpha}{2-\alpha}\phi_\eta^{\alpha -2}(X_t^i - X_t^j) + \frac{1}{2-\alpha}\phi_\eta^{2-\alpha}(X_t^i - X_t^k) \Bigr] \, dt\\
&= 2\gamma\alpha(N-1) \sum_{i=1}^N \sum_{j \neq i} A_{ij}.
\end{align*}
Substituting this estimate, \eqref{est:Ito-term} and~\eqref{pfzl} in~\eqref{pfzm} we find
\[
  \bigl(\alpha^2 \sigma^2 - 2\gamma\alpha(N-1)\bigr) \sum_{i=1}^N \sum_{j \neq i} A_{ij}
  = \alpha \delta \sum_{i=1}^N \sum_{j \neq i} A_{ij} 
  \leq C_0.
\]
We estimate $C_0$ by using $\phi(x) \leq |x|+1$ and $2\gamma(N-1)< \sigma^2$ as
\begin{align*}
  C_0 
  &= 2 \sqrt 2 (N-1) \frac{N^2}{N_-} \Bigl[\frac1{N} \sum_{j=1}^N \E \phi(X_0^j) + \Big( \sigma^2 + \frac {\gamma (N-1)}2 \Big) T\Bigr] \\
  &\leq \frac{2 \sqrt 2}{\beta_0} N (N-1) \Bigl[ 1 + \frac1N \sum_{j=1}^N \E |X_0^j| +  \frac{5}4 \sigma^2 T\Bigr].
\end{align*}
Putting this together, we obtain estimate~\eqref{c:ex:est} from
\begin{align*}
  \limsup_{\eta \to 0} \sum_{i=1}^N \sum_{j \neq i} A_{ij} 
  \leq \frac{C_0}{\alpha \delta}
  \leq \frac{2 \sqrt 2}{\alpha \delta \beta_0} N (N-1) \Bigl[ 1 + \frac1N \sum_{j=1}^N \E |X_0^j| +  \frac{5}4 \sigma^2 T\Bigr].
\end{align*}
\end{proof}

Let $C([0,\infty); (\R^2)^N)$ be endowed with the compact-open topology (uniform convergence on $[0,T]$ for any $T > 0$).

\begin{lem}[Tightness of $\bX^{\e,\ell}$ in $\e$ {\cite[Lemma 13]{FournierJourdain17}}] \label{l:tight}
Let the setting in Lemma \ref{l:ex:apri} be given. Then, for any $\ell \geq 1$ the sequence $(\bX^{\e, \ell})_\e$ is tight in $C([0,\infty); (\R^2)^N)$.
\end{lem}

\begin{proof} 
By the compact-open topology on $C([0,\infty); (\R^2)^N)$ it is enough to show that $(X^{i,\e, \ell})_\e$ is tight in $C([0,T]; \R^2)$ for all $T > 0$ and all $1 \leq i \leq N$, i.e.\ for all $\delta > 0$ there exists a compact $\cK \subset C([0,T]; \R^2)$ such that $\P (X^{i,\e, \ell} \notin \cK) < \delta$ for all $\e > 0$. 

Fix $T,\delta > 0$ and $1 \leq i \leq N$.
By relabelling the signs and the particles, we may assume $i=1$. By definition of the SDE in \eqref{SDE:e:ell}, we obtain
\begin{equation*}
  X_t^{1,\e,\ell} = X_0^{1} + \sigma B_t^1 + J_t^\e, \qquad J_t^\e :=  \gamma \int_0^t \Phi_\ell (\bX_s^{\e,\ell}) \sum_{j =1 }^N b^1 b^j K_\e (X_s^{1,\e,\ell} - X_s^{j,\e,\ell}) ds.
\end{equation*}
Since $X_0^{1}$ (as a constant-in-time path in $C([0,T]; \R^2)$) and $B^1$ are tight, it remains to show that the drift term $J^\e$ is tight. We do this with the set 
\[
    \cK := \{ f \in C([0,T]; \R^2) : f(0) = 0 \text{ and } \forall \, t,s : |f(t) - f(s)| \leq A |t-s|^\beta \}, 
\]
where $A > 0$ and $\beta \in (0,1)$ are constants which we choose later. Note that $\cK$ is equi-continuous and bounded and therefore pre-compact in $C([0,T]; \R^2)$.

To estimate $\P (J^\e \notin \cK)$ from above, we take in Lemma \ref{l:ex:apri}\ref{l:ex:apri:al} $\alpha = \frac12 (1 + 2 \gamma \sigma^{-2})$, take any $0 \leq s < t \leq T$ and compute (recall $\| \Phi_\ell \|_\infty \leq 1$ and $|K_\e (x)| \leq |x|^{-1}$) by applying H\"older's inequality with $p = 2 -\alpha > 1$
\begin{align} 
  |J_t^{\e} - J_s^{\e}|
  &\leq \gamma \sum_{j =2 }^N \int_s^t |X_r^{1,\e,\ell} - X_r^{j,\e,\ell}|^{-1} \, dr \notag \\
  &\leq \gamma \sum_{j =2 }^N ( t-s )^{\frac{1-\alpha}{2-\alpha}} \bigg( \int_s^t |X_r^{1,\e,\ell} - X_r^{j,\e,\ell}|^{\alpha-2} \, dr \bigg)^{\frac1{2-\alpha}}  \notag \\
  &\leq \gamma (t-s)^{\frac{1-\alpha}{2-\alpha}} \underbrace{ \sum_{j =2 }^N \bigg( 1 + \int_0^T |X_r^{1,\e,\ell} - X_r^{j,\e,\ell}|^{\alpha-2} \, dr \bigg) }_{=: Z^{\e}}.
  \label{est:control-of-J-for-tightness}
\end{align}
By Lemma \ref{l:ex:apri}\ref{l:ex:apri:al}, we have $\E Z^{\e} \leq C$ for some constant $C > 0$ independent of $A$ and $\e$. Using this together with the Markov Inequality, we find with $\beta = \frac{1-\alpha}{2-\alpha}$ that 
\comm{[proof of 1st ineq below: go by inclusion of events, and then realize that it follows from $Z^\e(\omega) \leq \frac A\gamma \implies J^{\e} \in \cK$, which is trivial.]}
\begin{align*}
   \P (J^{\e} \notin \cK)
   \leq \P \Big(Z^{\e} > \frac A\gamma\Big)
   \leq \frac\gamma A \E Z^{\e}
   \leq C \frac \gamma A.
 \end{align*} 
Since $A > 0$ is free to choose and $C$ is independent of $A$ and $\e$, we conclude that $J^\e$ is tight in $C([0,T];\R^2)$. This completes the proof of Lemma~\ref{l:tight}.
\end{proof}

The next step is to take the limit $\e\to0$ in the collection of processes $\{ \bX^{\e,\ell} \}_{\ell=1}^\infty$. This is the content of the following lemma.

\begin{lem}[Passing to $\e \to 0$ {\cite[Lemma 14]{FournierJourdain17}}] \label{l:eto0}
Let the setting in Theorem \ref{t:ex} be given and assume $N_+ \geq N_- \geq 1$. Then, there exists a probability space and a filtration $(\cF_t)_{t \geq 0}$ on which there is defined:
\begin{itemize}
  \item a $2N$-dimensional Brownian motion $\bB$ adapted to $(\cF_t)_{t \geq 0}$ such that the initial condition $\bX_0$ is $\cF_0$-measurable,
  \item for each $\ell \geq 1$ a process $\bX^\ell$ adapted to $(\cF_t)_{t \geq 0}$ which satisfies for all $t \geq 0$
  \begin{equation*}
    X_t^{i,\ell} = X_0^{i} + \sigma B_t^i + \gamma \int_0^t \Phi_\ell (\bX_s^\ell) \sum_{j = 1}^N b^i b^j K (X_s^{i,\ell} - X_s^{j,\ell}) ds
    \qquad \text{for all } 1 \leq i \leq N.
  \end{equation*}
\end{itemize}
Moreover, for all $\ell \geq 1$, $\bX^\ell$  is exchangeable with respect to $\bb$ and satisfies the same estimates as $\bX^{\e, \ell}$ in Lemma \ref{l:ex:apri}.
Finally, 
   for all $\ell' \geq \ell$, almost surely $\bX_t^\ell = \bX_t^{\ell'}$ for all $t \in [0, \tau^\ell)$, where
  \[
    \tau^\ell := \inf \Big\{ t \geq 0 : \Dsamel_t \leq \frac1\ell \Big\}
  \]
  and where $\Dsamel_t$ is defined as in \eqref{sD} with respect to $\bX_t^\ell$.
\end{lem}

Lemma \ref{l:eto0} is a paraphrasing of \cite[Lemma 14]{FournierJourdain17}. The arguments of the proof do not rely on the specific expression of the SDE \eqref{SDE:e:ell}; the main ingredients are Lemmas \ref{l:ex:apri} and \ref{l:tight}. To motivate this, we briefly outline the proof of \cite[Lemma 14]{FournierJourdain17}. It relies on tightness (Lemma \ref{l:tight}) of $(\bX^{\e,\ell})_\e$ to construct $\bX^\ell$ as the limit of a subsequence of $(\tilde \bX^{\e,\ell})_\e$, which is a special $\omega$-parametrization of $(\bX^{\e,\ell})_\e$ as constructed by the Skorokhod Representation Theorem. Care is needed to make sure that $\tilde \bB$ (which is also constructed by the Skorokhod Representation Theorem) is independent of $\ell$. By this construction, the stated properties on $\bX^\ell$ are inherited from those of $\bX^{\e,\ell}$; see in particular Lemma \ref{l:ex:apri}. 
\comm{[Proof that exchangeability with respect to $\bb$ is inherited (FJ17 says this without proof, and doesnt have the $\bb$ issue. MAP is OK with omitting the following proof): $\bX_t^\ell$ being exchangeable can be characterized from its law $\mu_t^\ell$ as follows: for any $\varphi \in C_b((\R^2)^N)$ and any permutation $\sigma \in S_N$ for which $\bb^\sigma = \bb$, 
\[
  \int [\varphi(\bx) - \varphi(\bx_\sigma)] \, d\mu_t^\ell(\bx) = 0.
\]
This follows directly from the exchangeability of the laws  $\mu_t^{\e,\ell}$ of $\bX_t^{\e,\ell}$, because we can just take $\psi(\bx) := \varphi(\bx) - \varphi(\bx_\sigma)$ as the test function in the narrow topology, and pass to the limit $\e \to 0$.]
}

\begin{proof}[Proof of Theorem \ref{t:ex}]
As above we may assume that $N_+ \geq N_- \geq 1$.
Let $\bB$ and $(\bX^\ell)_\ell$ be as in Lemma \ref{l:eto0}. The construction of $\bX$ from $(\bX^\ell)_\ell$ is done in detail in \cite[Theorem 7]{FournierJourdain17}. We briefly recall the construction here. Fix $\ell \geq 1$. For all $t \in [0, \tau^\ell)$, $\Phi_{\ell} (\bX_t^\ell) = 1$, and thus $\bX^\ell$ satisfies \eqref{SDE:full:param} until the stopping time $\tau^\ell$. Hence, by setting $\bX_t(\omega) := \bX_t^\ell(\omega)$ for all $t \in [0, \tau^\ell(\omega))$, all properties of Theorem~\ref{t:ex} are satisfied up to $\tau^\ell$. Since $\tau^\ell$ is nondecreasing in $\ell$, we can extend this construction of $\bX$ up to the $\ell$-independent stopping time
\[
  \tau := \sup_{\ell \geq 1} \tau^\ell.
\]
This completes the construction of $\bX$ up to $\tau$.
\map{I've changed some $T$'s below into $T\wedge \tau$'s, because we don't have the process $\bX$ after time $\tau$.}

To conclude that $\bX$ satisfies Theorem \ref{t:ex}, it remains to show that $\P (\tau < \infty) = 0$. Indeed, if so, then~\eqref{pfzs} and \eqref{c:ex:est} follow from Lemmas \ref{l:ex:apri} and \ref{l:eto0}, and \eqref{pfzb} follows from 
\comm{[Last equality follows from the events being equal. It follows from a trail of trivial steps; p.63.]}
\begin{equation} \label{pfyv}
  0 = \P(\tau < \infty) 
  = \lim_{T \to \infty}  \P(\tau < T) 
  = \lim_{T \to \infty} \P \Big( \inf_{0 \leq t < T\wedge \tau} \Dsame_t = 0 \Big).
\end{equation}

To show that $\P (\tau < \infty) = 0$, we follow the lines of the proof of \cite[Lemma 15]{FournierJourdain17}. This may seem more complex than necessary because in our  setting with signed particles almost half of all particle interactions or more are repulsive instead of attractive. For $N \geq 3$, however, unlike in the proof of Theorem \ref{t:coll}, it is difficult to isolate two particles with the same sign from the others, because any other particle of opposite sign is strongly attracted to them. Therefore, it is natural to consider collisions between 3 or perhaps even 4 particles. In fact, the argument below will consider \textit{all}  kinds of particle collision (except those of two particles with opposite sign).  

To encode any kind of particle collision, we set $I \subset \{1,\ldots,N\}$ as a candidate set of particles which collide at the same time-space point $(t,x)$. Let $n := |I| \geq 2$ be the number of particles in~$I$, and set 
\[
  X^I := \frac1n \sum_{i \in I} X^i, \qquad B^I := \frac1n \sum_{i \in I} B^i
\]
as averages over $i\in I$. To measure whether we are close to an $n$-particle collision, we set
\[
  R^I := \frac1{\sigma^2} \sum_{i \in I} |X^i - X^I|^2.
\]
Note that when $N=2$, $R^I$ coincides with the squared Bessel process $R$ \map{from Section \ref{s:intro:N2}}. Also, note that an $n$-particle collision at $t$ is characterized as the event $\{ R_t^I = 0 \}$. 

Using \eqref{pfyv} we translate the statement $\P (\tau < \infty) = 0$ in terms of $R^I$: 
\begin{equation} \label{pfzf} 
  \P(\tau < \infty) 
  = \lim_{T \to \infty} \P \Big( \exists \, 1 \leq i < j \leq N \text{ such that } b^i = b^j \text{ and } \inf_{t \in [0,T\wedge \tau)} R_t^{\{i,j\}} = 0 \Big).
\end{equation}
We will show that the probability in the right-hand side is $0$ for any $T > 0$. 

Instead of $I = \{i,j\}$ in \eqref{pfzf}, it turns out to be much easier to prove that for the full set $I = \{1, \ldots, N\}$ 
\begin{equation} \label{pfze}
  \P \Big( \inf_{t \in [0,T\wedge \tau)} R_t^I = 0 \Big) = 0.
\end{equation}
Indeed, from a computation which we perform below, it follows that $R^I$ is a squared Bessel process with dimension larger than $2$. Then, \eqref{pfze} follows from Theorem \ref{t:BP}\ref{t:BP:geq2}.

Next we derive \eqref{pfzf} from \eqref{pfze} by backward induction. The induction statement over $n$ is:
\begin{equation} \label{pfzd}
  \P \Big( \inf_{t \in [0,T\wedge \tau)} R_t^I = 0 \Big) = 0
  \qquad \text{for all } I \subset \{1,\ldots,N\} \text{ with } |I| = n.
\end{equation}

As preparation for proving the backward induction step, we take $I \subset \{1,\ldots,N\} $ with $|I| = n$ as given, and write
\[
  d X_t^i = \sigma dB_t^i + f_i(\bX_t) dt, \qquad
  f_i(\bX_t) := \gamma \sum_{j =1 }^N b^i b^j K(X_t^i - X_t^j) dt,
  \qquad f_I := \frac1n \sum_{i \in I} f_i.
\]
Then, 
\[
  d (X_t^i - X_t^I) = \sigma d (B_t^i - B_t^I) + (f_i(\bX_t) - f_I(\bX_t)) dt.
\]
Applying  It\^o's lemma to $R^I$ (e.g.~\cite[Th.~4.2.1]{Oksendal03}) 
we obtain the equation 
\begin{multline} \label{pfzh}
  d R_t^I 
  = \frac2{\sigma^2} \sum_{i \in I} (X_t^i - X_t^I) \cdot (f_i(\bX_t) - f_I(\bX_t)) dt
    + 2 (n-1) dt
    + \frac2\sigma \sum_{i \in I} (X_t^i - X_t^I) \cdot d (B_t^i - B_t^I)\\
    \text{which holds for }0\leq t < \tau.
\end{multline}
 
Next we simplify \eqref{pfzh}. We start with the stochastic term. Since $\sum_{i \in I} (X_t^i - X_t^I)  = n X_t^I - n X_t^I = 0$, we get 
\comm{[Proof: simple computation; p.63]}
\[
  \frac2\sigma \sum_{i \in I} (X_t^i - X_t^I) \cdot d (B_t^i - B_t^I)
  = 2 \sqrt{R_t^I} d \beta_t,
  \qquad \beta_t := \sum_{i \in I} \int_0^t \frac{  (X_s^i - X_s^I) }{\sqrt{ \sum_{j \in I} |X_s^j - X_s^I|^2 }} \cdot d B_s^i, 
\]
where $\beta_t$ is a $1$-dimensional Brownian Motion. 
For the drift term, we observe from $\sum_{i \in I} (X_t^i - X_t^I)  = 0$ that the part containing $f_I(\bX)$ vanishes. For the part containing $f_i(\bX)$, we substitute the definition of $f_i(\bX)$ and split the sum over $j$ in $j \in I$ and $j \notin I$. The part of the sum over $j \in I$ can be simplified; from the antisymmetry of $K$ and $x \cdot K(x) = 1$ for all $x \neq 0$ (recall $K(0) = 0$) we obtain
\begin{align*}
  &\frac{2 \gamma}{\sigma^2} \sum_{i,j \in I} b^i b^j (X_t^i - X_t^I) \cdot K(X_t^i - X_t^j) \\
  &= \frac\gamma{\sigma^2} \sum_{i,j \in I} b^i b^j (X_t^i - X_t^I) \cdot K(X_t^i - X_t^j) 
    +  \frac\gamma{\sigma^2} \sum_{i,j \in I} b^j b^i (X_t^j - X_t^I) \cdot K(X_t^j - X_t^i) \\
  &= \frac\gamma{\sigma^2} \sum_{i,j \in I} b^i b^j (X_t^i - X_t^j) \cdot K(X_t^i - X_t^j) \\
  &= \frac\gamma{\sigma^2} \Big( \sum_{i,j \in I} b^i b^j - \sum_{i \in I} (b^i)^2 \Big)
  = \frac\gamma{\sigma^2} \Big( \Big(\sum_{i \in I} b^i\Big)^2 - n \Big),
\end{align*}
which is independent of $\bX_t$. In particular, $b^I := \sum_{i \in I} b^i$ is the net sign of the $n$ particles in $I$. Putting this together, \eqref{pfzh} reads
\begin{equation} \label{pfzg} 
  d R_t^I 
  = 2 \sqrt{R_t^I} d \beta_t    
    + \Big( 2 (n-1) + \gamma \frac{(b^I)^2 - n}{\sigma^2} \Big) dt    
    + \frac{2\gamma}{\sigma^2} \sum_{i \in I} \sum_{j \notin I} b^i b^j (X_t^i - X_t^I) \cdot K(X_t^i - X_t^j) dt.
\end{equation}
Note that if $I = \{1,\ldots,N\}$, then the $\bX$-dependent part of the drift in \eqref{pfzg} vanishes, and thus $R^I$ is a squared Bessel process of dimension $2 (N-1) + \gamma ((b^I)^2 - N)\sigma^{-2}$. For general $I$, if the $\bX$-dependent part of the drift term in \eqref{pfzg} would not be there, then $R^I$ is a squared Bessel process with dimension $M := 2 (n-1) + \gamma ((b^I)^2 - n)\sigma^{-2}$. If $n \geq 3$, then (using $\sigma^2 > 2\gamma $ and $(b^I)^2 \geq 0$)
\[
  M = n \Big(2 - \frac\gamma {\sigma^2} \Big) - 2 + \frac{ \gamma (b^I)^2 }{ \sigma^2 }
  > \frac92 - 2 = \frac52,
\]
which is larger than $2$. In addition, if $I = \{i,j\}$ with $b^i = b^j$ (i.e.\ as in \eqref{pfzf}), then $M = 2 + 2\gamma \sigma^{-2} > 2$ as well. 

It is left to show that the $\bX$-dependent part of the drift in \eqref{pfzg} leaves invariant the property of a Bessel process with dimension greater than $2$ that it never reaches $0$. We refer for this to \cite[Lemma 15]{FournierJourdain17}, which applies with obvious modifications, and finish here with a brief outline of the argument. Suppose the induction statement \eqref{pfzd} holds for $n+1$. Take any $I$ with $|I| = n$, and suppose that $R_t^I$ reaches $0$ at some $t_0 >0$. Then, for any $k \notin I$, since $R_t^{I \cup \{k\}}$ stays away from $0$, $X_t^k$ stays away from $(X_t^i)_{i \in I}$. This can be used to gain sufficient control on the $\bX$-dependent part of the drift in \eqref{pfzg} to derive that $R_t^I$ behaves like a squared Bessel process for $t$ close to $t_0$. Since we have already computed above that the dimension of this squared Bessel process is greater than $2$, we reach a contradiction with $R_{t_0}^I = 0$. The estimates on the $\bX$-dependent part of \eqref{pfzg} on which \cite{FournierJourdain17} relies only require that the absolute value of the summand can be controlled, and thus the appearance of $b^i b^j$ is of little importance. In addition, the exchangeability with respect to $\bb$ plays no significant role. 
\end{proof}

\newpage

\section{The mean-field limit as $N\to\infty$}
\label{s:MFlimit}

In this section we state and prove Theorem \ref{t:mf-limit}.
We describe the signed particles as position-sign couples $Y^i := (X^i,b^i)\in \R_\pm^2$ (recall Section \ref{s:prel}). We give $\R_\pm^2$ the product topology; a continuous curve in this space must preserve the value of $b$, which holds for the solutions that we consider in this paper. The differential operators  $\nabla$, $\div$, and $\Delta$ only operate on the $x$-variable. Similarly to Section \ref{s:global-existence} we equip $C([0,\infty),\cP(\R^2_\pm))$ with the topology of uniform convergence on compact time intervals and weak convergence in $\cP(\R^2_\pm)$; with this topology the space is Polish~\cite[Ex.~IV.2.2]{Conway90}.  

Next we describe the assumptions of Theorem \ref{t:mf-limit}. Let $\sigma>0$ be fixed and $\gamma=\gamma_N$ depend on $N$. We assume that there exists $\ol\gamma> 0$ such that 
\begin{equation}
\label{cond:params-MF-limit}
N \gamma_N \to \ol \gamma \qquad \text{as }N\to \infty,\qquad \text{with}\qquad \sigma^2>  2\ol \gamma.
\end{equation}
Note that $\gamma_N = O(\frac1N)$, which corresponds to a mean-field scaling of the interaction between particles. Also, \eqref{cond:params-MF-limit} implies that there exists $0<\alpha<1$ and $c>0$ such that 
\begin{equation}
  \label{cond:params-MF-limit-2}
  \alpha \sigma^2 - 2\gamma_N (N-1) \geq c \qquad \text{for  all $N$ large enough}.
\end{equation}

For the initial condition, instead of drawing $\bX_0^N$ from a certain distribution for fixed $N$, we draw $\bY_0^N = \bigl((X^{N,i}_0,b^{N,i})\bigr)_{i=1,\dots,N}$ from an exchangeable distribution on $(\R^2_\pm)^{N}$ for each $N \geq 2$. In particular, this makes the signs $\bb^N$ random. We further assume that the initial positions $(X^{N,i}_0)_{i=1,\dots,N}$ are almost surely distinct and that 
\begin{equation}
  \label{cond:params-MF-limit-3}
\sup_{N\geq 2} \frac1N \sum_{i=1}^N \E |X^{N,i}_0| < \infty.
\end{equation}
We also assume that $\bb^N$ cannot be too skewed; concretely, we assume that there exists $\beta_0>0$ such that for all $N \geq 2$ 
\begin{equation}
  \label{cond:params-MF-limit-4}
\min \{ N_+, N_- \}
\geq \beta_0 
\qquad \text{almost surely},
\end{equation}
where, as in Theorem \ref{t:ex}, $N_+$ ($N_-$) are the number of positive (negative) signs in $\bb^N$.
Note that we do not assume that $(Y^{N,i}_0)_{i=1,\dots,N}$ are independent, and  if necessary, therefore, we can fix $N_+$ and $N_-$. 

Under these assumptions, the following corollary of Theorem~\ref{t:ex} holds.

\begin{cor}[Theorem~\ref{t:ex} for $\bY$] \label{c:ex:random-signs}
Let $\bY_0 = (\bX_0, \bb)$ be as above. Then, there exists a global weak solution $\bY = (\bX,\bb)$ to \eqref{SDE:full:param}. Moreover, this solution is exchangeable, $\E |\bX_t| < \infty$ for all $t > 0$, and \eqref{pfzs}, \eqref{pfzb} and the following version of \eqref{c:ex:est} holds
\begin{equation}
  \label{c:ex:est:Sec5}
  \frac1{N(N-1)} \sum_{\substack{i,j=1\\i\not=j}}^N \E \int_0^T |X_t^i - X_t^j|^{\alpha - 2} dt
  \leq
  C,
  \end{equation}
where $C > 0$ is independent of $N$.
\end{cor}

\begin{proof}
Since $\bY_0$ is exchangeable, the conditional law of $\bX_0$ given $\bb$ is exchangeable with respect to $\bb$. Corollary~\ref{c:ex:random-signs} then follows from applying Theorem \ref{t:ex} to each realisation of $\bb$.
\end{proof}

From Corollary \eqref{c:ex:random-signs} we obtain from the sequence $(\bY^N_0)_{N \geq 2}$ of independent initial conditions as described above a sequence of independent weak solutions $(\bY^N)_{N \geq 2}$ to \eqref{SDE:full:param} and a sequence of independent, corresponding Brownian motions $(\bB^N)_{N \geq 2}$. To describe the limit $N\to\infty$, we consider in Theorem~\ref{t:mf-limit} below the random empirical measures
\[
L^N(dxdb) 
:= \frac1N \sum_{i=1}^N \delta_{(X^{N,i}, b^{N,i})}(dxdb)\quad \in \cP(C([0,\infty),\R^2_\pm))  
\]
on path space. The desired limit of $L^N$ is a $\cP(C([0,\infty),\R^2_\pm))$-valued random variable $L$. We write $L_t$ for the marginal of $L$ corresponding to time $t$, which is a $\cP(\R^2_\pm)$-valued random variable. We aim to show that $L = (L_t)_{t\geq 0}$ is almost surely a solution of the following partial differential equation for $\rho: [0,\infty)\to \cP(\R^2_\pm)$:
\begin{equation}
\label{eq:PDE}
  \partial_t \rho_t + \ol \gamma \div \rho_t \sK\rho_t = \sigma \Delta \rho_t,
\end{equation}
where we write for $\mu\in \cP(\R^2_\pm)$ and $(x,b)\in \R^2_\pm$
\[
(\sK\mu)(x,b) := \int_{x'\in\R^2} K(x-x') \big( \mu(dx',\{b\}) - \mu(dx',\{-b\}) \big).
\]
Separated into densities $\rho^+ = \rho(\cdot,\{+1\})$ and $\rho^- = \rho(\cdot,\{-1\})$ of positive and negative particles this equation reads
\begin{align*}
  &\partial_t \rho^+_t + \ol\gamma \div \bigl[\rho^+_t K{*}(\rho^+_t-\rho^-_t)\bigr] = \sigma \Delta \rho^+_t,\\
  &\partial_t \rho^-_t - \ol\gamma \div \bigl[\rho^-_t K{*}(\rho^+_t-\rho^-_t)\bigr] = \sigma \Delta \rho^-_t.
\end{align*}

\begin{rem}
  By a standard argument, the equation~\eqref{eq:PDE} can be interpreted as characterizing a `nonlinear SDE', also known as a McKean-Vlasov process, an interpretation that we do not pursue here. 
  \end{rem}

\begin{defn}
  \label{d:solution-of-PDE}
  Fix $\rho^\circ\in \cP(\R^2_\pm)$.
$\rho\in C([0,\infty),\cP(\R^2_\pm))$ is a weak solution of the partial differential equation~\eqref{eq:PDE} with initial datum $\rho^\circ$ if it satisfies
\begin{enumerate}[]
\item  \label{d:solution-of-PDE-initial-data}
$\rho_{t=0} = \rho^\circ$;
\item   \label{d:solution-of-PDE-cond-power-minus-one=part} 
For every $T>0$, 
\begin{equation}
  \label{d:solution-of-PDE-cond-power-minus-one} 
\int_0^T \iint\limits_{(\R^2_\pm)^2}|x-x'|^{-1} \rho_t(dxdb)\rho_t(dx'db')\, dt < \infty;
\end{equation}
\item\label{d:solution-of-PDE-def-W} For every $T > 0$ and every $\varphi\in C_b^2(\R^2_\pm)$ we have $\cW[T,\varphi](\rho)=0$, where $\cW[T,\varphi]$ is defined by
\begin{align}
\cW[T,\varphi](\rho) &:= \int_{\R^2_\pm} \varphi(x,b) \bigl[\rho_T(dxdb) - \rho_0(dxdb)\bigr] \notag\\
&\qquad - \ol\gamma \int_0^T \iint\limits_{(\R^2_\pm)^2} bb' K(x-x') \cdot \nabla \varphi(x)\rho_t(dxdb)\rho_t(dx'db')\, dt\notag\\
&\qquad {}- \sigma \int_0^T \int_{\R^2_\pm} \Delta\varphi(x,b) \rho_t(dxdb)\, dt.
\label{eqdef:W}
\end{align}
\end{enumerate}
\end{defn}

\begin{thm}[Mean-field limit $N\to\infty$]
  \label{t:mf-limit}
Consider the setup above; assume in particular condition~\eqref{cond:params-MF-limit}. Assume that $L^N_{t=0}$ converges narrowly in probability to some deterministic $\rho^\circ\in \cP(\R^2_\pm)$.  Then:
\begin{enumerate}[label=(\roman*)]
\item\label{t:mf-limit:tightness} The sequence $L^N$ of random variables is tight in~$\cP(C([0,\infty);\R^2_\pm))$. Therefore there exists a subsequence $N_k\to\infty$ such that $L^{N_k}$ converges in distribution to a (possibly random) limit $L\in \cP(C([0,\infty);\R^2_\pm))$.
\item\label{t:mf-limit:solution} Defining $\rho$ to be the random variable in $C([0,\infty);\cP(\R^2_\pm))$ given by $\rho_t := L_t$, almost surely $\rho$ is a weak solution to~\eqref{eq:PDE} with initial datum $\rho^\circ$, as in Definition~\ref{d:solution-of-PDE}.
\end{enumerate}
\end{thm}

\begin{rem}
In contrast to many mean-field convergence results, we do not prove that the limit $L$ is deterministic. This is related to the question of uniqueness of weak solutions of~\eqref{eq:PDE}, which is still open; if uniqueness for~\eqref{eq:PDE} is proved, then a classical argument gives convergence to a deterministic limit. As for the uniqueness question, in the more widely-studied Patlak-Keller-Segel problem for the fully attractive case, uniqueness is known to hold for solutions with finite $L\log L$-norm~\cite{Egana-FernandezMischler16} or with certain time-dependent bounds on an $L^p$-norm~\cite{BedrossianMasmoudi14}, but the general case is still open.
\end{rem}

\subsection{Proof of Theorem~\ref{t:mf-limit}}
  The proof of this theorem is a relatively straightforward generalisation of~\cite[Theorem~6]{FournierJourdain17}. We describe the main steps and refer for some details to~\cite{FournierJourdain17}. 

Part~\ref{t:mf-limit:tightness}, the tightness of $L^N$, is a consequence of the following lemma by Sznitman. 
\begin{lem}{\cite[p.~178]{Sznitman91}}
  \label{l:Sznitman}
  For every $N\geq 1$, let $\{Z^{N,i}$, $i=1,\dots,N\}$ be an exchangeable set of random variables in a Polish space $\cZ$. Then the following are equivalent:
  \begin{enumerate}
    \item $\displaystyle \Big(\frac1N\sum_{i=1}^N \delta_{Z^{N,i}}\Big)_{N}$ is a tight sequence of random variables in $\cP(\cZ)$;
    \item $(Z^{N,1})_{N}$ is a tight sequence of random variables in $\cZ$.
  \end{enumerate}
\end{lem}
\noindent
We apply this lemma to $Y^{N,i} \in \cZ := C([0,\infty); \R^2_\pm)$ and use the following lemma to  conclude that $L^N$ is a tight sequence of random variables in $\cP(\cZ) = \cP(C([0,\infty); \R^2_\pm))$. 

\begin{lem}
  \label{l:tightness-XN1}
$\{Y^{N,1}, N\geq 2\}$ is a tight sequence of random variables in~$C([0,\infty);\R^2)$. 
\end{lem}

\begin{proof}[Proof of Lemma~\ref{l:tightness-XN1}]
The argument is similar to the proof of Lemma~\ref{l:tight}. It is sufficient to prove tightness in $C_b([0,T];\R^2)$ for each $T>0$. Writing 
\[
X^{N,1}_t = X_0^{N,1} + \sigma B_t^{N,1} + J^{N,1}_t, 
\qquad
J^{N,1}_t = \gamma_N \int_0^t \sum_{j=2}^N b^1b^j K(X^{N,1}_s-X^{N,j}_s)\, ds,
\]
we note that $\{B^{N,1}: N\geq 2\}$ is tight in $C_b([0,T];\R^2)$ by the properties of Brownian motion. In addition, since $L^N_0$ converges narrowly in probability to $\rho^\circ$ and therefore is tight in $\cP(\R^2_\pm)$, the sequence $Y^{N,1}$ is tight in $\R^2_\pm$ by Lemma~\ref{l:Sznitman}. We therefore only need to show that $J^{N,1}$ is tight.

By exchangeability and our assumptions on the parameters in \eqref{cond:params-MF-limit-2}, \eqref{cond:params-MF-limit-3} and \eqref{cond:params-MF-limit-4}, the estimate~\eqref{c:ex:est:Sec5} implies
\begin{equation}
\label{est:tightness-N-to-infty}  
\sup_{N\geq 2}\E Z_N <\infty, \qquad \text{where}\qquad Z_N = \frac1{N-1} \sum_{\substack{j=2}}^N  \int_0^T |X_s^{N,1} - X_s^{N,j}|^{\alpha - 2} \,ds.
\end{equation}
As in the proof of Lemma~\ref{l:tight} we then derive that 
\[
|J^{N,1}_t-J^{N,1}_s| \leq |t-s|^{\frac{1-\alpha}{2-\alpha}} \;\gamma_N \sum_{j=2}^N \biggl( 1+ \int_0^T   |X^{N,1}_t - X^{N,j}_t|^{\alpha-2}\, dt\biggr)
\leq |t-s|^{\frac{1-\alpha}{2-\alpha}}  \, \frac{\sigma^2}2 (1+Z_N)
\]
for $N$ large enough, so that the tightness of $J^{N,1}$ in $C([0,T];\R^2)$ follows by~\eqref{est:tightness-N-to-infty} and the Markov inequality. 
\end{proof}

\bigskip
We turn to part~\ref{t:mf-limit:solution}. Let $L$ be as in part~\ref{t:mf-limit:tightness} and $\rho\in C([0,\infty);\cP(\R^2_\pm))$ with $\rho_t = L_t$ as in part~\ref{t:mf-limit:solution}. We note that $\rho_{t=0}$ is the weak limit in $\cP(\R^2_\pm)$ of $L^N_{t=0}$, and by assumption coincides with $\rho^\circ$. This shows that $\rho$ satisfies part~\ref{d:solution-of-PDE-initial-data} of Definition~\ref{d:solution-of-PDE}.
\comm{[Proof of $L^N_{t=0} \xto d L_{t=0}$. First, we have $L^N \xto d L$. If the evaluation map $e_t : \cP( C([0,\infty);\R^2_\pm) ) \to \cP(\R^2_\pm)$ given by $e_t(\ell) = \ell_t$ (here, $\ell \in \cP( C([0,\infty);\R^2_\pm) )$ is a deterministic path; the same object as $L(\omega)$ for a given $\omega$) is (narrowly) continuous, then $L^N_t = e_t(L_N) \xto d e_t(L) = L_t$ is a standard thm (follows from the pushforward and change of variables formula). The continuity of $e_t$ goes as follows: we need to show that $\ell^\e \weakto \ell$ implies $e_t(\ell^\e) = \ell_t^\e \weakto \ell_t = e_t(\ell)$. Verifying this is easy when writing the narrow convergences out in terms of test functions; see p.61 for details]} 

Part~\ref{d:solution-of-PDE-cond-power-minus-one=part} of Definition~\ref{d:solution-of-PDE} is the following lemma.
\begin{lem}[{\cite[p.~2824-2825]{FournierJourdain17}}]
  \label{l:FJ-3}
For all $T>0$,
\[
\E \int_0^T \iint\limits_{(\R^2_\pm)^2} |x-x'|^{\alpha-2}L_t(dxdb)L_t(dx'db') dt < \infty.
\]
\end{lem}
\noindent
The proof follows from remarking that for any $m>0$
\begin{multline*}
\E\int_0^T \iint\limits_{(\R^2_\pm)^2} (|x-x'|^{\alpha-2}\wedge m) L^N_t(dxdb)L^N_t(dx'db')\, dt \\
= \frac {mT}N + \E\frac1{N^2} \sum_{i=1}^N \sum_{j \neq i} \int_0^T (|X_t^{N,i}- X_t^{N,j}|^{\alpha-2}\wedge m)  \,dt.
\end{multline*}
The expectation on the right-hand side can be controlled by~\eqref{c:ex:est:Sec5}, uniformly in $m$ and $N$.
This shows that $\rho$ also satisfies part~\ref{d:solution-of-PDE-cond-power-minus-one=part}. 
\comm{[This argument relies on
\[ \E\int_0^T \iint\limits_{(\R^2_\pm)^2} (|x-x'|^{-1}\wedge m) L^N_t(dxdb)L^N_t(dx'db')\, dt
\to \E\int_0^T \iint\limits_{(\R^2_\pm)^2} (|x-x'|^{-1}\wedge m) L_t(dxdb)L_t(dx'db')\, dt. \]
To show this, it is OK to take $\E$ inside $\int_0^T$. 
Since $L^N \to L$ in distribution and $L \to (L,L)$ is continuous, we also have $(L^N , L^N) \to (L , L)$ in distribution. Then, similar to the orange comment above, $(L_t^N , L_t^N) \to (L_t , L_t)$ in distribution, and that is already stronger than what we need to show (as we only need to show it for a specific test function)].
}

\medskip
Finally, to show that $\rho$ also satisfies part~\ref{d:solution-of-PDE-def-W} of Definition~\ref{d:solution-of-PDE} we use the following two lemmas. 
\begin{lem}[{\cite[p.~2825--2826]{FournierJourdain17}}]
  \label{l:FJ-1}
  Fix $T$ and $\varphi$ as in part~\ref{d:solution-of-PDE-def-W} of Definition~\ref{d:solution-of-PDE}. 
For any $\e>0$ define $\cW_\e$ as in~\eqref{eqdef:W} but with $K$ replaced by $K_\e(x) = x/(|x|^2 +\e^2)$ as in~\eqref{eqdef:K_e}. Then there exists $C >0$ independent of $N, \e$ such that 
\[
\E|\cW(L)-\cW_\e(L)| 
+ \sup_{N\geq 2}\E |\cW(L^N)-\cW_\e(L^N)| \leq C\e^{1-\alpha}. 
\]
\end{lem}

\begin{proof}
Combining the observation
\begin{equation*}
  |K(x)-K_\e(x)| = \frac{\e^2}{|x|(|x|^2 + \e^2)}
\end{equation*}
with the following application of Young's inequality,
\begin{equation*}
  |x|^{1-\alpha} \e^{1+\alpha} 
  \leq \frac{1-\alpha}2 |x|^2 + \frac{1+\alpha}2 \e^2
  \leq |x|^2 + \e^2,
\end{equation*}
we obtain $|K(x)-K_\e(x)|\leq \e^{1-\alpha}|x|^{\alpha-2}\bOne\{x\not=0\}$. With this estimate, Lemma \ref{l:FJ-1} follows from Lemma~\ref{l:FJ-3} and~\eqref{c:ex:est:Sec5}. 
\end{proof}

\begin{lem}[{\cite[p.~2825]{FournierJourdain17}}]
  \label{l:FJ-2}
There exists $C > 0$ independent of $N$ such that 
\[
\E | \cW(L^N) | \leq \frac C{\sqrt N} + C' (N \gamma_N - \ol \gamma).  
\]
\end{lem}

\begin{proof}
Applying It\^o's lemma (e.g.~\cite[Th.~4.2.1]{Oksendal03}) to $\frac1N \sum_{i=1}^N \varphi (X_t^{N,i},b^{N,i})$ yields
\begin{align*}
  \cW(L^N) 
  &= \frac\sigma N\sum_{i=1}^N \int_0^T \nabla \varphi(X_t^{N,i},b^{N,i}) \cdot dB^{N,i}_t \\
  &\quad + \frac{ N \gamma_N - \ol \gamma}{N^2} \int_0^T \sum_{i,j=1}^N b^{N,i} b^{N,j} K(X_t^{N,i} - X_t^{N,j}) \cdot \nabla \varphi(X_t^{N,i}, b^{N,i}) \, dt.
\end{align*}
Then we take the expectation of the absolute value of both terms in the right-hand side. The second term is bounded by \eqref{c:ex:est:Sec5}. For the square of the first term, we apply Jensen's inequality and estimate the result using independence and It\^o isometry as
\begin{align*}
\MoveEqLeft\E\biggl[ \Big( \frac\sigma N \sum_{i=1}^N \int_0^T \nabla \varphi(X_t^{N,i},b^{N,i}) \cdot dB^{N,i}_t \Big)^2 \biggr] \\
&= \frac{\sigma^2}{N^2} \sum_{i,j=1}^N \E\biggl[  
  \int_0^T \nabla \varphi(X_t^{N,i},b^{N,i}) \cdot dB^{N,i}_t
  \int_0^T \nabla \varphi(X_t^{N,j},b^{N,j}) \cdot dB^{N,j}_t\biggr]\\
&= \frac{\sigma^2}{N^2} \sum_{i=1}^N \E\biggl[ \Big(
  \int_0^T \nabla \varphi(X_t^{N,i},b^{N,i}) \cdot dB^{N,i}_t \Big)^2 \biggr]\\
&= \frac{\sigma^2}{N^2} \sum_{i=1}^N   \E\biggl[\int_0^T \big|\nabla \varphi(X_t^{N,i},b^{N,i})\big|^2 dt\biggr] \leq \frac{\sigma^2}N T \|\nabla \varphi\|_\infty^2.
\qedhere
\end{align*}
\end{proof}

Finally we complete the proof of Theorem \ref{t:mf-limit}. We expand
\begin{align*}
\E |\cW(L)| &\leq \underbrace{\E|\cW(L)-\cW_\e(L)|}_{(a)}
  + \underbrace{\E|\cW_\e(L)-\cW_\e(L^N)|}_{(b)}
  + \underbrace{\E|\cW_\e(L^N)-\cW(L^N)|}_{(c)}
  + \underbrace{\E|\cW(L^N)|}_{(d)}.
\end{align*}
From Lemma~\ref{l:FJ-1} we deduce that terms $(a)$ and $(c)$ are bounded by $C\e^{1-\alpha}$ uniformly in $N$. Taking $N\to\infty$, we observe from the continuity of $\cW_\e$ that term $(b)$ vanishes and from Lemma~\ref{l:FJ-2} that term $(d)$ also vanishes. Then, taking $\e \downarrow 0$ we obtain $\E|\cW(L)|=0$. Hence, $\cW(L)=0$ almost surely. This proves  part~\ref{t:mf-limit:solution} of Theorem \ref{t:mf-limit} and concludes the proof.

\section*{Funding}
PvM has received financial support from JSPS KAKENHI Grant Numbers JP20K14358 and JP24K06843. 
MP's visit to Kanazawa university was supported by Prof.\ M.\ Kimura's JSPS KAKENHI Grant Number JP20KK0058. 

\bibliographystyle{alphainitials} 
\bibliography{refsPatrick,refsMark}

\end{document}